\documentclass[12pt,a4paper,twoside]{extarticle}

\usepackage[T1]{fontenc}
\usepackage{amsmath,amssymb,mathtools,amsthm}
\usepackage{aliascnt}
\IfFileExists{newtxtext.sty}
  {\usepackage{newtxtext,newtxmath}}
  {\usepackage{mathptmx}}
\usepackage{microtype}
\usepackage[a4paper,inner=13mm,outer=12mm,top=11mm,bottom=13mm,headsep=3mm,footskip=6mm]{geometry}
\usepackage{array,booktabs,longtable,tabularx}
\usepackage{enumitem}
\usepackage{titlesec}
\usepackage{fancyhdr}
\usepackage[numbers,sort&compress]{natbib}
\usepackage{hyperref}
\usepackage[nameinlink,noabbrev]{cleveref}
\usepackage{tikz-cd}
\usepackage{xcolor}
\usepackage{caption}
\usepackage{etoolbox}
\usepackage{csquotes}

\definecolor{linkblue}{RGB}{22,62,105}
\hypersetup{
  colorlinks=true,
  linkcolor=linkblue,
  citecolor=linkblue,
  urlcolor=linkblue,
  pdftitle={Guarded Realization Semantics: Occurrence-Sensitive Certificates and Behavior-Dependent Lower Bounds},
  pdfauthor={SeungJu Lee}
}

\setlist{nosep,leftmargin=1.45em,labelsep=.45em}
\setlist[enumerate,1]{label=\arabic*.,ref=\arabic*}
\setlist[itemize,1]{label=\textbullet}
\titleformat{\section}{\normalfont\bfseries\fontsize{10}{11}\selectfont}{\thesection}{.55em}{}
\titleformat{\subsection}{\normalfont\bfseries\fontsize{8.7}{9.7}\selectfont}{\thesubsection}{.5em}{}
\titleformat{\subsubsection}{\normalfont\itshape\fontsize{8.2}{9.1}\selectfont}{\thesubsubsection}{.45em}{}
\titlespacing*{\section}{0pt}{1.6ex plus .4ex minus .2ex}{.55ex}
\titlespacing*{\subsection}{0pt}{1.05ex plus .25ex minus .15ex}{.35ex}
\titlespacing*{\subsubsection}{0pt}{.8ex plus .2ex minus .1ex}{.25ex}

\newtheoremstyle{compactplain}{2pt}{2pt}{\itshape}{}\bfseries{.}{.45em}{}
\newtheoremstyle{compactdef}{2pt}{2pt}{\normalfont}{}\bfseries{.}{.45em}{}
\theoremstyle{compactplain}
\newtheorem{theorem}{Theorem}[section]
\newaliascnt{proposition}{theorem}
\newtheorem{proposition}[proposition]{Proposition}
\aliascntresetthe{proposition}
\newaliascnt{lemma}{theorem}
\newtheorem{lemma}[lemma]{Lemma}
\aliascntresetthe{lemma}
\newaliascnt{corollary}{theorem}
\newtheorem{corollary}[corollary]{Corollary}
\aliascntresetthe{corollary}
\theoremstyle{compactdef}
\newaliascnt{definition}{theorem}
\newtheorem{definition}[definition]{Definition}
\aliascntresetthe{definition}
\newaliascnt{construction}{theorem}
\newtheorem{construction}[construction]{Construction}
\aliascntresetthe{construction}
\newaliascnt{example}{theorem}
\newtheorem{example}[example]{Example}
\aliascntresetthe{example}
\newaliascnt{remark}{theorem}
\newtheorem{remark}[remark]{Remark}
\aliascntresetthe{remark}

\crefname{theorem}{theorem}{theorems}
\Crefname{theorem}{Theorem}{Theorems}
\crefname{proposition}{proposition}{propositions}
\Crefname{proposition}{Proposition}{Propositions}
\crefname{lemma}{lemma}{lemmas}
\Crefname{lemma}{Lemma}{Lemmas}
\crefname{corollary}{corollary}{corollaries}
\Crefname{corollary}{Corollary}{Corollaries}
\crefname{definition}{definition}{definitions}
\Crefname{definition}{Definition}{Definitions}
\crefname{construction}{construction}{constructions}
\Crefname{construction}{Construction}{Constructions}
\crefname{example}{example}{examples}
\Crefname{example}{Example}{Examples}
\crefname{remark}{remark}{remarks}
\Crefname{remark}{Remark}{Remarks}

\makeatletter
\renewenvironment{proof}[1][\proofname]{\par\pushQED{\qed}\normalfont
  \topsep2pt\trivlist\item[\hskip\labelsep\bfseries #1\@addpunct{.}]\ignorespaces}
 {\popQED\endtrivlist\@endpefalse}
\makeatother

\newcommand{\PSh}{\operatorname{PSh}}
\newcommand{\Pres}{\mathsf{Pres}}
\newcommand{\Beh}{\mathsf{Beh}}

\newcommand{\Err}{\mathsf{Err}}

\newcommand{\Sub}{\operatorname{Sub}}
\newcommand{\Hom}{\operatorname{Hom}}
\newcommand{\Aut}{\operatorname{Aut}}
\newcommand{\Lan}{\operatorname{Lan}}
\newcommand{\Ran}{\operatorname{Ran}}
\newcommand{\colim}{\operatorname*{colim}}
\newcommand{\dom}{\operatorname{dom}}
\newcommand{\trk}{\operatorname{trk}}
\newcommand{\Stab}{\operatorname{Stab}}
\newcommand{\minsub}{\min_{\subseteq}}
\newcommand{\Npos}{\mathbb N_{>0}}
\newcommand{\down}{\mathord{\downarrow}}
\newcommand{\blank}{[-]}
\newcommand{\iso}{\cong}
\newcommand{\mono}{\hookrightarrow}
\newcommand{\onto}{\twoheadrightarrow}
\newcommand{\esssup}{\operatorname*{ess\,sup}}

\newcommand{\compactpart}[2]{%
  \par\vspace{1.5ex}\noindent\rule{\linewidth}{.3pt}\par
  \noindent{\bfseries\fontsize{9.4}{10.4}\selectfont #1. #2}\par
  \vspace{.45ex}\noindent\rule{\linewidth}{.3pt}\par\vspace{.6ex}}

\title{\vspace{-7mm}\bfseries\fontsize{15}{16}\selectfont
Guarded Realization Semantics\\[-1mm]
\normalfont\fontsize{9.5}{10.5}\selectfont
Occurrence-Sensitive Certificates and Behavior-Dependent Lower Bounds}
\author{SeungJu Lee}
\date{\today}

\begin{document}
\maketitle
\vspace{-5mm}

\begin{abstract}
Distinct proofs, programs, formulas, or rewrite paths may have the same
observable behavior while differing in their occurrence structure, sharing,
interfaces, or transformation history.  We develop a guarded realization
semantics that retains these distinctions when an error is extracted.  The
resulting magnitude is bounded above by a certificate attached to the chosen
realization and below by the greatest behavior-dependent lower bound valid for
every compatible error realization.

For linear double-pushout rewriting in the typed presheaf setting, we identify
the greatest subobject transported intact through a rewrite step and through a
finite rewrite path.  Sound guarded local estimates compose to give pathwise
upper certificates.  At the set level, the complementary lower bound is the
infimum of the magnitudes in a behavior fiber.  For non-discrete categories,
when the magnitude is functorial and the relevant pointwise right Kan extension
exists, the lower reflection is given by that extension; it reduces to the
strict-fiber infimum under a Grothendieck fibration hypothesis.

For continuous surjective linear observations from complex normed spaces onto
finite-dimensional complex normed spaces, the lower reflection is the induced
quotient norm.  Applied to finitely many distinct characters on a compact
metrizable abelian group, this gives an interpolation norm with an exact dual
formula.  Under the hypotheses of the continuous and discrete Abel transfer
theorems, these norms bound tail amplitudes from observed coefficients.  With
the corresponding boundary-agreement assumptions, the bounds yield Mellin and
generating-function consequences, as well as an application to the normalized
point-count error of smooth, projective, geometrically connected curves of
positive genus over finite fields.  Altogether, every realization to which the relevant guards and soundness
hypotheses apply satisfies
\[
  Q\bigl(O(\Err(r))\bigr)
  \leq A(\Err(r))
  \leq U(r).
\]
\end{abstract}



\section{Introduction}\label{sec:introduction}

Distinct proofs, programs, formulas, or rewrite paths can produce the same
observable result.  One program may reuse an intermediate value while another
recomputes it; one proof may share a subproof while another duplicates it; and
two rewrite paths may reach isomorphic targets through different intermediate
objects.  These differences can be ignored when only the final behavior
matters.  They cannot always be ignored when an error estimate depends on the
occurrences used by a realization, on the domain where a rule is valid, or on
the path along which local estimates are composed.

The framework used here therefore extracts an error before passing to observed
behavior.  A realization \(r\) has a structured error object \(\Err(r)\), which
may retain a residual function, an occurrence attribution, or a validity
domain.  The functor \(O\) records the observed error behavior, and \(A\)
assigns a numerical magnitude.  If \(U(r)\) is a sound upper certificate on
the actual guard of \(r\), and \(Q\) is a sound lower bound determined by the
observation, then
\[
  Q\bigl(O(\Err(r))\bigr)
  \leq A(\Err(r))
  \leq U(r).
\]
The two endpoints use different information.  The upper endpoint may depend
on the occurrences and rule path of the chosen realization.  The lower
endpoint must remain valid for every error realization compatible with the
same observation.  In particular, the lower construction does not select or
reconstruct a preferred realization.

Guarded local semantics supplies the numerical upper estimates, while rewrite
structure determines where the occurrence labels used by those estimates
remain intact.  The local certificates take values in a unital ordered monoid.
If the identity certificate is sound, each generating certificate is sound on
its guard, and soundness is preserved by ordered composition on the composite
guard, their product gives a sound certificate for the entire path.

Occurrence transport is modeled by linear double-pushout rewriting in
\[
  \mathcal C=[\mathcal H,\mathbf{Set}]/\Sigma ,
\]
where \(\mathcal H\) is a small incidence category and \(\mathcal C\) is a
Grothendieck topos.  Whole-grain incidence models motivate this presentation
\cite{kock-petri}; the topos and adhesive-rewriting results used in its
analysis are drawn from
\cite{maclane-moerdijk,lack-sobocinski,behr-harmer-krivine,
zanasi-free-hypergraph,bonchi-rewriting-frobenius}.
For a linear DPO derivation with a monic match, its context is the greatest
subobject of the source transported intact through the step.  Equivalently, an
occurrence is transported intact precisely when its intersection with the
matched left-hand side is contained in the retained interface.  Pullback
composition of the track spans gives the greatest intact domain of a finite
rewrite path.  On the rewrite bicategory freely generated by the chosen
parallel-independence diamonds, these spans extend to a pseudofunctor; see
\cref{thm:maximal-step,thm:path-domain,thm:track-pseudo}.  An occurrence that
meets nonretained material may have a transformed descendant, but it is not
transported as the same intact subobject.  Moreover, track transport alone
does not preserve a numerical certificate: the certificate assignment must
also be invariant under occurrence isomorphisms and compatible with the
context embeddings, as required in \cref{prop:certificate-transport}.

The passage from realizations to behavior also places a precise restriction on
reification.  Proof identity, sharing, and intensional distinctions provide
the motivating examples
\cite{girard,strassburger,hasegawa,pfenning,hu-pientka}.
Let \(\mathcal E\) be a small realization category and
\(p:\mathcal E\to\mathcal B\) a behavior functor.  If the Yoneda embedding
\(y_{\mathcal E}\) factors, up to natural isomorphism, through \(p\), then
\(p\) is faithful and any two realizations with isomorphic behaviors are
isomorphic.  Consequently, a behavior functor that identifies nonisomorphic
realizations cannot support such a representable factorization.  This is an
obstruction to representable behavior-level reification, not to every form of
coding.

There are two relevant alternatives.  If \(\mathcal E\) is a small groupoid
and \(\mathcal B\) is a discrete behavior set, the behaviorwise left-Kan
aggregate is generally nonrepresentable, but its category of elements
recovers the full realization fiber.  If a realization fibration is presented
by a pseudofunctor with small fibers, or if a fixed universe enlargement is
used, fiberwise presheaf completion represents each realization separately in
its own fiber and transports these representables by reindexing.  Thus the
aggregate and fiberwise constructions retain realization data without choosing
one representative of each behavior.  The categorical tools involved are the
standard Yoneda, Kan-extension, fibration, and descent constructions
\cite{maclane,streicher,coquand-mannaa-ruch}.

The numerical effect of passing to behavior has a simple universal description
for sets.  Given functions
\[
  O:E\longrightarrow B,
  \qquad
  A:E\longrightarrow[0,\infty],
\]
define
\[
  Q_A(b)=\inf\{A(e):O(e)=b\},
  \qquad \inf\varnothing=\infty.
\]
Then \(Q_A\circ O\leq A\), and every other function
\(H:B\to[0,\infty]\) satisfying \(H\circ O\leq A\) also satisfies
\(H\leq Q_A\).  Hence \(Q_A\) is the greatest lower bound determined by
behavior alone; this is \cref{thm:behavior-reflection}.

The non-discrete formula requires more than replacing sets by categories.  If
\[
  O:\mathcal E\longrightarrow\mathcal B,
  \qquad
  A:\mathcal E\longrightarrow[0,\infty]
\]
are functors and the pointwise right Kan extension exists, then
\[
  (\Ran_OA)(b)
  =
  \inf_{(u:b\to O(e))\in(b\downarrow O)} A(e).
\]
Thus the value is controlled in general by a comma category, not by the strict
fiber alone.  When \(O\) is a Grothendieck fibration, a cartesian lift of
\(u:b\to O(e)\) produces an object over \(b\), and functoriality of \(A\)
reduces the comma-category infimum to
\[
  (\Ran_OA)(b)=\inf_{e\in\mathcal E_b}A(e).
\]
The strict-fiber formula in \cref{thm:fibrational-reflection} is therefore a
consequence of the fibration hypothesis rather than a formula for arbitrary
observation functors.

A first analytic realization of this lower bound is obtained from quotient
norms.  Let \(X\) be a complex normed space, let \(B\) be a
finite-dimensional complex normed space, and let
\[
  T:X\twoheadrightarrow B
\]
be continuous and linear.  Then
\[
  q_T(b)=\inf\{\|x\|_X:Tx=b\}
\]
is the quotient norm induced by \(T\), and standard quotient-space duality
\cite{conway,folland} gives
\[
  q_T(b)
  =
  \sup_{\varphi\in B^*,\,\varphi\ne0}
  \frac{|\varphi(b)|}{\|T^*\varphi\|_{X^*}}.
\]
For a nonempty compact convex uncertainty set \(U\subseteq B\), Sion's minimax
theorem \cite{sion} gives the corresponding robust formula.  A different
compactness statement is used for infinite systems of tests: when
\(X=Y^*\), the tests are weak-star continuous, and every finite subsystem is
solvable with one common norm bound, weak-star compactness gives a simultaneous
solution with the same bound.  This finite-intersection argument is separate
from both minimax and categorical descent.

Compact-group interpolation gives a concrete family of quotient norms.  Let
\(G\) be a compact metrizable abelian group and let
\(\chi_1,\ldots,\chi_m\) be distinct characters.  Their coefficient map
from \(L^\infty(G)\) onto \(\mathbb C^m\) is a continuous linear
surjection, so the least \(L^\infty\)-norm of a function with prescribed
coefficients is a quotient norm.  Its dual norm is the \(L^1\)-norm of the
corresponding character polynomial, yielding the formula in
\cref{thm:compact-duality}.  Character orthogonality and compact-group harmonic
analysis provide the underlying tools, while Kronecker--Weyl averaging
identifies the orbit groups associated with continuous and discrete
frequencies \cite{bailleul,rudin,katznelson}.

These interpolation norms become tail lower bounds through Abel limits.  In
the continuous case, let \(F\in L^1_{\mathrm{loc}}([0,\infty))\) satisfy
\[
  \int_0^\infty |F(y)|e^{-\varepsilon y}\,dy<\infty
  \qquad(\varepsilon>0).
\]
If the Abel coefficient limits exist at finitely many distinct real
frequencies, the essential tail amplitude of \(F\) is at least the associated
compact-orbit interpolation norm.  In the discrete case, the corresponding
conclusion holds for a sequence \((F_n)\) satisfying
\[
  \sum_{n\geq0}|F_n|r^n<\infty
  \qquad(0<r<1),
\]
provided that the Abel limits exist at pairwise distinct phases.  These are
the continuous and discrete bounds of
\cref{thm:continuous-abel,thm:discrete-abel}.

Mellin transforms and generating functions supply such coefficients only with
an additional boundary guard.  A simple Mellin boundary pole may be used when
the meromorphic continuation agrees with the original Mellin integral on the
interior side of the boundary; similarly, a simple generating-function
boundary pole may be used when the continuation agrees there with the original
power series.  Meromorphic continuation by itself does not identify the Abel
coefficients of the original error term.  The Mellin conventions are those of
\cite{zagier}.  The discrete result is also applied to a smooth, projective,
geometrically connected curve of genus at least one over a finite field,
using the point-count formula and Frobenius eigenvalue bound from
\cite{weil,deligne}.  In that application the coefficient-constrained lower
bound is kept distinct from the exact limsup of the specified Frobenius
trigonometric realization.

The upper and lower constructions meet only after the error extractor,
magnitude, and observation have been fixed.  The general comparison theorem
assumes a realization category fibered in groupoids, with small fibers or a
fixed universe enlargement and with occurrence structure transported by
cartesian arrows; a cartesian or suitably isomorphism-invariant error
extractor; sound guarded local certificates; an existing behaviorwise
reflection, with functorial magnitude in the non-discrete case; and the
norming or lower-bound hypotheses required by any test family used to compute
that reflection.  Under these conditions, every realization to which the
guards apply satisfies the displayed lower--upper comparison
\cref{thm:integrated}.  The compact-orbit specialization additionally assumes
the corresponding continuous or discrete Abel coefficients and that the
chosen error magnitude dominates the relevant tail amplitude
\cref{thm:summary}.

Parts~I and II develop occurrence transport and the distinction between
realizations and behavior.  Part~III constructs compositional upper
certificates and behavior-dependent lower bounds.  Part~IV develops the
compact-group interpolation and Abel applications, and Part~V combines the two
estimates.

\

\setcounter{tocdepth}{2}
\vspace{-1ex}
\tableofcontents
\vspace{.5ex}

\compactpart{I}{Realizations, Specifications, and Occurrence Structure}

\section{Specifications, realizations, and error objects}\label{sec:realization-error}

\subsection{The realization--error data}
\begin{definition}[Realization--error datum]
A realization--error datum consists of categories and functors
\[
\begin{tikzcd}[column sep=3.4em]
\mathcal R \arrow[r,"\Err"] \arrow[d,"p"'] & \mathcal E \arrow[d,"e"] \arrow[r,"O"] & \mathcal B_{\mathrm{err}}\\
\mathcal S \arrow[r,equal] & \mathcal S &
\end{tikzcd}
\]
The displayed diagram is required to commute; in particular, $e\circ\Err=p$.  The datum also includes a magnitude map $A:\operatorname{Ob}(\mathcal E)\to[0,\infty]$.  When categorical transport of magnitudes is used, $A$ is additionally required to be a functor from $\mathcal E$ to the poset category $[0,\infty]$ with its usual order.  Here $\mathcal S$ is the category of specifications, $p:\mathcal R\to\mathcal S$ is the realization family, $\mathcal E$ is the category of error objects over specifications, $\Err$ extracts the error of a realization relative to its specification, and $O$ records an observable error behavior.  For $r\in\mathcal R$ we write
\[
  b_r:=O(\Err(r)),\qquad a_r:=A(\Err(r)).
\]
The error object may retain a domain, residual function, interval, distribution, occurrence attribution, or other information about the defect; $A$ records only its magnitude.
\end{definition}

\begin{construction}[Error as comparison with a target]
Suppose $u:\mathcal M\to\mathcal S$ is a semantic bundle, $J:\mathcal R\to\mathcal M$ satisfies $uJ=p$, and $\tau:\mathcal S\to\mathcal M$ is a chosen target section.  If a comparison operation over $\mathcal S$,
\[
  \Delta:\mathcal M\times_{\mathcal S}\mathcal M\longrightarrow\mathcal E,
  \qquad e\Delta=\operatorname{pr}_{\mathcal S},
\]
is available, define
\[
  \Err(r):=\Delta\bigl(J(r),\tau(p(r))\bigr).
\]
This covers normed differences $J(r)-\tau(p(r))$, equation residuals, approximate-program outputs, proof obligations left open by a derivation, and partial errors defined only on a guard domain.
\end{construction}

\begin{definition}[Guard]
A guard is a specified subcollection $\mathcal G\subseteq\operatorname{Ob}(\mathcal E)$, or more generally a subobject of the state/error domain, on which an observation or quantitative inequality is defined and sound.  Each partial rule, boundary transfer, or error estimate is understood on its actual guard.  Analytic continuation, syntactic applicability, and domain membership are different guards and must not be conflated.
\end{definition}

\subsection{Upper and lower certificates}
\begin{definition}[Occurrence-based upper certificate]
A function $U:\operatorname{Ob}(\mathcal R)\to[0,\infty]$ is a sound upper certificate when
\[
  A(\Err(r))\le U(r)
\]
for every guarded realization under consideration.  Such certificates are normally computed compositionally from rule occurrences and therefore may distinguish realizations with the same external behavior.
\end{definition}

\begin{definition}[Observation-forced lower semantics]
A function $Q:\operatorname{Ob}(\mathcal B_{\mathrm{err}})\to[0,\infty]$ is a sound lower semantics when
\[
  Q(O(e))\le A(e)
\]
for every guarded error realization $e$.  The optimal such function is developed in \cref{sec:quant-reflection}.
\end{definition}

\begin{proposition}[Certification sandwich]\label{prop:basic-sandwich}
If $U$ is a sound upper certificate and $Q$ a sound lower semantics, then every guarded realization satisfies
\[
  Q\bigl(O(\Err(r))\bigr)\le A(\Err(r))\le U(r).
\]
If the left endpoint exceeds a proposed upper certificate, the specification, observation data, or certificate assumptions are inconsistent.  Equality of the endpoints proves quantitative optimality relative to the chosen observation $O$.
\end{proposition}
\begin{proof}
The two inequalities are the defining soundness conditions.  The consequences follow from transitivity of the order on $[0,\infty]$.
\end{proof}

\begin{remark}[When occurrence structure matters]
The lower term depends only on observed error behavior, but the construction of $\Err(r)$ and $U(r)$ may depend on sharing, copying, intermediate values, residual domains, and rewrite history.  Collapsing realizations to behavior before extracting error can therefore discard occurrence or domain information needed for a correct upper bound even when it leaves the eventual lower reflection unchanged.
\end{remark}
\section{Presentations, behaviors, and occurrences}\label{sec:five-roles}

\subsection{Proof presentations}
A \emph{presentation} (in particular, a proof or algorithm presentation) is a finite or effectively finitely described derivational object carrying enough occurrence data to represent a derivation graph. Depending on the model, it may contain proposition or value occurrences, inference occurrences with ordered finite input and output ports, explicit sharing, nested regions, binding or discharge links, cycles or guarded recursive references, and external interfaces.  No single formalism for derivation graphs is assumed; finiteness is imposed only where a later statement explicitly requires it.

\subsection{The five roles}
\begin{definition}[Presentation]
A presentation is a concrete derivational object $P\in\Pres$. It records occurrences and incidence, not merely output behavior.
\end{definition}
\begin{definition}[Behavior]
A behavior is an observational class of presentations. It may be represented by a surjection $q:\Pres\onto\Beh$ or, categorically, by a functor from presentations to behaviors. Distinct presentations may have the same behavior.
\end{definition}
\begin{definition}[Code]
A code is an object making a presentation available as input to another construction. A quotation operation has the form $\ulcorner-\urcorner:\Pres\to\mathsf{Code}$.
\end{definition}
\begin{definition}[Occurrence structure]
The occurrence structure records which uses are identical, which are distinct copies, and which inference or proposition occurrence is connected to which port.
\end{definition}
\begin{definition}[Source support]
Fix a finite set $S$ of external source occurrences and an output occurrence $t$. A support of $t$ is a subset $A\subseteq S$ sufficient for a derivation of $t$. The inclusion-minimal supports form an antichain
\[
\operatorname{Supp}_P(t)\subseteq\mathcal P_{\mathrm{fin}}(S).
\]
\end{definition}
Behavior records what is observed; code records which presentation is inspectable; occurrence structure records how it is assembled; support records which external sources suffice.  These roles are not interchangeable: behavior and support can forget occurrence data, while flattening can forget quotation boundaries.
\section{Typed whole-grain generative hypernetworks}
For whole-grain occurrence modelling and explicit finite-set incidence data, see \cite{kock-petri}.  The topos and adhesive-category facts invoked in \cref{prop:topos,thm:maximal-step,thm:track-pseudo} are given in \cite{maclane-moerdijk,lack-sobocinski}.

\subsection{The incidence schema}
Fix a small category $\mathcal H$ with a distinguished object $V$ and, for every pair $(n,m)\in\mathbb N^2$, an object $E_{n,m}$. For $1\le i\le n$ and $1\le j\le m$, include generating morphisms
\[
s_i:E_{n,m}\to V,\qquad t_j:E_{n,m}\to V,
\]
with no equations beyond the category axioms. Let
\[
\widehat{\mathcal H}:=[\mathcal H,\mathbf{Set}]
\]
be the category of covariant $\mathcal H$-sets. Equivalently, one may replace the incidence category by $\mathcal H^{\mathrm{op}}$ and use the standard presheaf convention; this is a change of presentation of the schema, not a claim that $[\mathcal H,\mathbf{Set}]$ and $[\mathcal H^{\mathrm{op}},\mathbf{Set}]$ are canonically equivalent for an arbitrary fixed $\mathcal H$.

An object $X\in\widehat{\mathcal H}$ consists of a set $X(V)$ of vertex occurrences, sets $X(E_{n,m})$ of $(n,m)$-hyperedge occurrences, and incidence maps $X(s_i),X(t_j):X(E_{n,m})\to X(V)$. Repeated vertices and parallel hyperedges remain distinct elements.

\begin{definition}[Typing signature]
A typing signature is an object $\Sigma\in\widehat{\mathcal H}$. The category of typed generative hypernetworks is the slice
\[
\mathcal C:=\widehat{\mathcal H}/\Sigma.
\]
An object $p_X:X\to\Sigma$ assigns types to vertices and operation labels to hyperedges compatibly with incidence.
\end{definition}

\begin{remark}
The input arity of one hyperedge, the total number of hyperedge occurrences in a finite network, and the length of a directed path are three different parameters.
\end{remark}

\begin{proposition}[Inherited categorical structure]\label{prop:topos}
The category $\mathcal C$ is a Grothendieck topos. In particular, it is complete, cocomplete, extensive, and adhesive, and every subobject poset $\Sub(X)$ is a complete Heyting algebra.
\end{proposition}
\begin{proof}
A functor category over a small category is a Grothendieck topos, and every slice of a Grothendieck topos is again a Grothendieck topos. Such categories are complete, cocomplete, and extensive; their subobject posets are complete Heyting algebras. Topoi are adhesive.
\end{proof}

\subsection{Occurrences and subnetworks}
A monomorphism $a:A\mono X$ in $\mathcal C$ is a \emph{subnetwork occurrence}. Two monomorphisms with the same domain can represent different occurrences of the same abstract pattern in $X$. Subobjects are equivalence classes of monomorphisms under isomorphism over $X$. For $A,B\le X$, write $A\wedge B$ for pullback intersection and $A\vee B$ for union in the Heyting algebra $\Sub(X)$. If $u:X\to Y$, pullback gives $u^*:\Sub(Y)\to\Sub(X)$.

\begin{definition}[Comparison modes]
Let $a:A\mono X$ and $b:B\mono Y$ be subnetwork occurrences.
\begin{enumerate}
\item They are strictly identical only when all data are literally equal.
\item Relative to a specified host comparison $\psi:X\iso Y$, they are isomorphic as occurrences when there is an isomorphism $\phi:A\iso B$ with $\psi a=b\phi$, compatible with the chosen interfaces.
\item They have a nontrivial common core when there is a monic span $A\leftarrow K\rightarrow B$ with $K$ noninitial; when needed, a greatest common core must be specified by an additional universal property.
\item They are rewrite related when joined by a zigzag in a chosen rewrite system.
\item They are transport related when one is the intact residual of the other along a rewrite path.
\end{enumerate}
\end{definition}
These notions answer different questions and should not be collapsed into one equivalence relation.
\section{DPO rewriting and the maximal preserved region}
For DPO and adhesive rewriting, see \cite{lack-sobocinski,behr-harmer-krivine,zanasi-free-hypergraph,bonchi-rewriting-frobenius}.

\subsection{Direct derivations and track spans}
A linear DPO rule in $\mathcal C$ is a span of monomorphisms
\[
p=(L\xleftarrow{\ell}K\xrightarrow{r}R).
\]
Given a monic match $m:L\mono G$, a direct derivation $\rho:G\Rightarrow H$ is a diagram
\[
\begin{tikzcd}[column sep=2.7em,row sep=1.7em]
K \arrow[r,"\ell"] \arrow[d,"k"'] & L \arrow[d,"m"] & K \arrow[l,"\ell"'] \arrow[r,"r"] \arrow[d,"k"'] & R \arrow[d,"n"]\\
D \arrow[r,"d"'] & G & D \arrow[l,"d"] \arrow[r,"h"'] & H
\end{tikzcd}
\]
in which both squares are pushouts. The object $D$ is the context or preserved part.

\begin{definition}[Track span]
The track span of $\rho$ is
\[
\trk(\rho):=(G\xleftarrow{d}D\xrightarrow{h}H).
\]
It represents a partial morphism from $G$ to $H$, defined on the image of $D$ in $G$. Both legs are monic for the linear DPO steps considered here.
\end{definition}

\subsection{Stable subobjects}
\begin{definition}[Stepwise stability]
A subobject $A\le G$ is $\rho$-stable if
\[
A\wedge L\le K
\]
as subobjects of $G$, where $L,K,D$ are identified with their monic images. Write $\Stab(\rho)\subseteq\Sub(G)$ for the stability poset.
\end{definition}
The condition says that $A$ may meet the matched left-hand side only in material retained by the rule interface.

\begin{theorem}[Maximal transportable region]\label{thm:maximal-step}
For a linear DPO derivation $\rho:G\Rightarrow H$ with context $D\mono G$, and every $A\le G$, the following are equivalent:
\begin{enumerate}
\item $A$ is $\rho$-stable;
\item $A\le D$;
\item $A\mono G$ factors through $d:D\mono G$.
\end{enumerate}
Consequently,
\[
\Stab(\rho)=\down d(D):=\{A\in\Sub(G):A\le d(D)\},
\]
and $D$ is the greatest subobject transported intact through $\rho$.
\end{theorem}
\begin{proof}
The left pushout square is also a pullback, so $L\wedge D=K$ in $\Sub(G)$; the pushout gives $L\vee D=G$. If $A\le D$, then $A\wedge L\le K$. Conversely, if $A\wedge L\le K$, distributivity gives
\[
A=A\wedge(L\vee D)=(A\wedge L)\vee(A\wedge D)\le D.
\]
Factorization through $d$ is the definition of subobject order.
\end{proof}

\begin{definition}[Deletion obstruction]
For $A\le G$, define
\[
\operatorname{DelObs}_{\rho}(A)\quad:\Longleftrightarrow\quad A\wedge L\nleq K.
\]
It vanishes if and only if $A$ is transportable intact.
\end{definition}

\begin{remark}
An ``obstruction object'' obtained by subtracting $K$ from $A\wedge L$ would require complements unavailable in a general Heyting algebra.  The displayed predicate is invariant under isomorphism of subobject representatives and holds exactly when intact transport fails.
\end{remark}

\subsection{Canonical transport}
\begin{proposition}[Canonical step transport]\label{thm:canonical-step}
A DPO step induces a canonical order isomorphism
\[
\rho_{\#}:\down d(D)\xrightarrow{\sim}\down h(D).
\]
If $a:A\mono G$ is stable and $\bar a:A\mono D$ is the unique factorization $d\bar a=a$, then $\rho_{\#}(A)$ is represented by $h\bar a:A\mono H$. The transport preserves all meets and joins internal to the principal ideals.
\end{proposition}
\begin{proof}
Composition with $d$ identifies $\Sub(D)$ with $\down d(D)$, and composition with $h$ identifies $\Sub(D)$ with $\down h(D)$. Their composite is the stated order isomorphism.
\end{proof}

\begin{corollary}[Uniqueness of intact residual]
A stable occurrence has a unique transported occurrence up to the unique isomorphism determined by its common factorization through $D$.
\end{corollary}
\begin{proof}
This is the uniqueness of the factorization through the monomorphism $d:D\mono G$, followed by the order isomorphism of \cref{thm:canonical-step}.
\end{proof}

\begin{example}
If a rewrite deletes a hyperedge $e$ and its private output vertex $v$ while preserving an interface vertex $k$, an occurrence containing only $k$ and material outside the redex lies below $D$ and transports canonically. An occurrence containing $e$ or $v$ has no intact residual. It may have a transformed descendant, but that is a rewrite of the subnetwork rather than transport of the same occurrence.
\end{example}
\section{Rewrite paths and greatest pathwise domains}

\subsection{The partial-map category}
Let $\mathbf{Par}(\mathcal C)$ be the category whose morphisms $X\rightharpoonup Y$ are isomorphism classes of spans $X\xleftarrow{u}M\xrightarrow{v}Y$ with $u$ monic; composition is by pullback. Each DPO step determines $\trk(\rho):G\rightharpoonup H$. For the linear steps here, both legs are monic, and this property is preserved by pullback composition.

\begin{definition}[Track map of a rewrite path]
For
\[
\pi=(G_0\xRightarrow{\rho_1}G_1\xRightarrow{\rho_2}\cdots\xRightarrow{\rho_n}G_n),
\]
define
\[
\trk(\pi):=\trk(\rho_n)\circ\cdots\circ\trk(\rho_1).
\]
Choose a representative $G_0\xleftarrow{u_\pi}D_\pi\xrightarrow{v_\pi}G_n$ with both legs monic. The subobject $u_\pi(D_\pi)\le G_0$ is the pathwise preserved domain.
\end{definition}

\begin{theorem}[Greatest pathwise transport domain]\label{thm:path-domain}
A subobject $A\le G_0$ transports intact through every step of a finite path $\pi:G_0\Rightarrow^*G_n$ iff $A\le u_\pi(D_\pi)$. Thus $u_\pi(D_\pi)$ is the greatest pathwise transport domain.
\end{theorem}
\begin{proof}
Induct on the path length. Appending a step with context span $G_n\xleftarrow dD\xrightarrow hG_{n+1}$ pulls $D_\pi$ back along $d$. The new domain consists of the points whose $\pi$-transport lies in $D$, equivalently those stable for every prior step and for the appended step.
\end{proof}

\begin{proposition}[Canonical path transport]\label{thm:path-transport}
Every path induces an order isomorphism
\[
\pi_{\#}:\down u_\pi(D_\pi)\xrightarrow{\sim}\down v_\pi(D_\pi).
\]
For composable paths $\pi,\sigma$, one has $(\sigma\circ\pi)_{\#}=\sigma_{\#}\circ\pi_{\#}$ on the greatest common domain on which the right-hand side is defined.
\end{proposition}
\begin{proof}
For the composite span, composition with $u_\pi$ and $v_\pi$ identifies $\Sub(D_\pi)$ with the two displayed principal ideals, as in the proof of \cref{thm:canonical-step}. The second assertion follows from pullback composition and uniqueness of factorization through the composite domain.
\end{proof}

\begin{corollary}[Finite decidability]
Assume a finite network means an $\mathcal H$-set with finite support in the arity objects and finitely many total occurrences, and assume rules and matches are given effectively. For a fixed finite rewrite path and finite subnetwork occurrence, intact transport is decidable.
\end{corollary}
\begin{proof}
Compute the composite track span by finitely many pullbacks and test componentwise inclusion of finite sets.
\end{proof}
\section{Coherence, higher relations, and partial monodromy}

\subsection{The rewrite bicategory}
Fix a set $\mathcal R$ of linear DPO rules. Let $\mathbf{Rew}_{\mathcal R}$ be the bicategory freely generated by direct $\mathcal R$-derivations as 1-cells and by specified invertible 2-cells for parallel-independence (local Church--Rosser) diamonds, subject to the usual horizontal and vertical coherence laws.

Let $\mathbf{ParSpan}(\mathcal C)$ be the bicategory whose objects are those of $\mathcal C$, whose 1-cells are spans $X\xleftarrow{u}M\xrightarrow{v}Y$ with $u$ monic, and whose 2-cells are isomorphisms of spans. Composition is by chosen pullbacks; different choices are connected by the canonical associativity and unit isomorphisms. Passing to the isomorphism classes in $\mathbf{Par}(\mathcal C)$ discards these 2-cells, so $\mathbf{ParSpan}(\mathcal C)$ is used as the target.

\begin{theorem}[Track pseudofunctor]\label{thm:track-pseudo}
The assignment of a linear DPO step to its track span extends to a pseudofunctor
\[
\mathbf{Trk}:\mathbf{Rew}_{\mathcal R}\longrightarrow\mathbf{ParSpan}(\mathcal C).
\]
For each specified parallel-independence diamond, the two composite track spans are canonically isomorphic.
\end{theorem}
\begin{proof}
Define $\mathbf{Trk}$ on paths by pullback composition. The pullback associators and unitors supply the pseudofunctor constraints. For a parallel-independence diamond, the proof of the DPO local Church--Rosser theorem constructs both residual orders from common pullback/pushout data \cite[Theorem~7.7]{lack-sobocinski}.  Since the relevant pushouts along monomorphisms are also pullbacks \cite[Lemma~4.3]{lack-sobocinski}, the pullback apex computing either composite track span is canonically isomorphic to the common object in that construction, compatibly with its legs to the source and common target.  The two composite track spans are therefore canonically isomorphic.  Because $\mathbf{Rew}_{\mathcal R}$ is freely generated by these specified diamonds subject to bicategorical coherence, the assignment extends uniquely with the pullback associators and unitors as pseudofunctor constraints.
\end{proof}

\begin{corollary}[Coherence of independent transport]
If two parallel-independent rewrites form a local Church--Rosser diamond with common result $H$ up to a specified isomorphism, any subobject in the common domain of the two composite track spans is transported to isomorphic occurrences in $H$, independently of rewrite order.
\end{corollary}
\begin{proof}
The specified diamond is sent by \cref{thm:track-pseudo} to an isomorphism of the two composite track spans; restricting that isomorphism to the chosen subobject gives the comparison in $H$.
\end{proof}

\subsection{Occurrences, coherence, and partial monodromy}
\begin{definition}[Transport bicategory of occurrences]
Define $\mathbf{Occ}_{\mathcal R}$ as follows.
\begin{enumerate}
\item Objects are pairs $(G,A)$ with $A\le G$.
\item A 1-cell $(G,A)\to(H,B)$ is a path $\pi:G\Rightarrow^*H$ such that $A\le\dom\trk(\pi)$ and $B=\pi_{\#}(A)$.
\item A 2-cell from $\pi$ to $\sigma$ is a 2-cell $\alpha:\pi\Rightarrow\sigma$ in $\mathbf{Rew}_{\mathcal R}$ whose span isomorphism $\mathbf{Trk}(\alpha)$ restricts to a commuting isomorphism between the chosen representatives of $\pi_{\#}(A)$ and $\sigma_{\#}(A)$.
\end{enumerate}
Composition is induced by path concatenation and \cref{thm:path-transport}.
\end{definition}

\begin{definition}[Transport equivalence and partial monodromy]
Two occurrences are transport equivalent when connected by a zigzag of 1-cells in $\mathbf{Occ}_{\mathcal R}$. If a rewrite path $\pi:G\Rightarrow^*H$ is equipped with an ambient isomorphism $i:H\iso G$, then the composite monic span
\[
G\xleftarrow{u_\pi}D_\pi\xrightarrow{i\,v_\pi}G
\]
induces an order isomorphism between the principal ideals $\down u_\pi(D_\pi)$ and $\down i(v_\pi(D_\pi))$. Such partial order automorphisms, composed on their pullback domains, constitute \emph{transport monodromy}. No global automorphism of $\Sub(G)$ is asserted unless the track domain is all of $G$.
\end{definition}

\begin{proposition}[Obstruction to coherent path independence]
Suppose $\pi,\sigma:G\Rightarrow^*H$ share a stable subobject $A\le G$. If there is a 2-cell $\alpha:\pi\Rightarrow\sigma$, then $\pi_{\#}(A)$ and $\sigma_{\#}(A)$ are isomorphic through the restriction of $\mathbf{Trk}(\alpha)$. Hence nonisomorphic transported occurrences obstruct such a coherence 2-cell.
\end{proposition}
\begin{proof}
Apply the track pseudofunctor and restrict the resulting span isomorphism to $A$.
\end{proof}
\section{Proof-theoretic specialization}
For proof nets, proof identity, sharing, and intensional distinctions, see \cite{girard,strassburger,hasegawa,pfenning,hu-pientka}.

\subsection{Deductive presentations}
Fix a typing signature $\Sigma$ and the corresponding slice $\mathcal C=\widehat{\mathcal H}/\Sigma$.
\begin{definition}[Deductive presentation]
A deductive presentation on $\mathcal C$ consists of:
\begin{enumerate}
\item a class $\operatorname{Corr}(\Gamma,\Delta)$ of correct open typed hypernetworks with premise interface $\Gamma$ and conclusion interface $\Delta$;
\item a set $\mathcal R_{\mathrm{cut}}$ of boundary-preserving linear DPO rules (or a separately specified DPOI formalism with the analogous track theorems);
\item preservation of correctness by every rule in $\mathcal R_{\mathrm{cut}}$;
\item stability of the open-net correctness predicate under the context embeddings and boundary reindexings used in intact transport.
\end{enumerate}
A member of $\operatorname{Corr}(\Gamma,\Delta)$ is a proof net of $\Delta$ from $\Gamma$, and
\[
\Gamma\vdash\Delta\quad\Longleftrightarrow\quad\operatorname{Corr}(\Gamma,\Delta)\ne\varnothing.
\]
\end{definition}

\begin{definition}[Subproof occurrence]
For $P\in\operatorname{Corr}(\Gamma,\Delta)$, a subproof occurrence is a monomorphism $a:A\mono P$ with the induced open boundary such that $A$ is correct for that boundary.
\end{definition}

\subsection{Residuals under cut elimination}
\begin{proposition}[Residual subproof theorem]
Let $\rho:P\Rightarrow Q$ be a correctness-preserving cut-elimination step and $A\mono P$ a subproof occurrence. If $A$ is $\rho$-stable, then it has a canonical residual $\rho_{\#}(A)\mono Q$. Its internal typed hypernetwork is unchanged, and its boundary incidences are transported through the context object.
\end{proposition}
\begin{proof}
By \cref{thm:maximal-step}, $A$ factors through the preserved context $D$. Composing with $D\mono Q$ gives the residual. The additional correctness-stability assumption ensures that the induced boundary reindexing preserves correctness of the open proof net.
\end{proof}

\begin{corollary}[Local obstruction]
A subproof occurrence containing a cut link or formula occurrence deleted by the chosen step has no intact residual under that step. It may instead rewrite to a new subproof; that is not transport of an unchanged occurrence.
\end{corollary}
\begin{proof}
Such an occurrence meets the matched left-hand side outside the retained interface, so it fails the stability criterion in \cref{thm:maximal-step}.
\end{proof}

\subsection{Normal-form residuals}
Ordinary confluence joins reduction paths but does not itself identify the histories of surviving occurrences.

\begin{definition}[Transport-coherent normalization]
For a terminating proof-net rewrite bicategory and a chosen normal form $N$ of $P$, let $\mathcal N(P,N)$ be the full subgroupoid of the hom-category on the reduction paths $P\Rightarrow^*N$, with invertible transport-compatible 2-cells as morphisms. The system is \emph{transport-coherent} if every $\mathcal N(P,N)$ is connected, and \emph{thinly transport-coherent} if every $\mathcal N(P,N)$ is contractible.
\end{definition}

\begin{proposition}[Path-independent residual in normal form]
Assume termination, correctness preservation, and transport coherence. If $A\mono P$ is stable along every reduction path to the chosen normal form $N$, then every normalization path transports $A$ to an occurrence isomorphic in $N$. Under thin transport coherence, this isomorphism is canonical.
\end{proposition}
\begin{proof}
For paths $\pi,\sigma$, choose an invertible 2-cell $\alpha:\pi\Rightarrow\sigma$ and restrict $\mathbf{Trk}(\alpha)$ to $A$. Thinness makes the resulting comparison independent of the chosen 2-cell.
\end{proof}

\begin{remark}
Allowing $A$ itself to rewrite calls for additional marked or coupled structure, for example a rewrite system on ambient proofs and marked subproofs or a fibration of marked rules over unmarked rules.
\end{remark}

\compactpart{II}{Separation, Reification, and Transport}

\section{What extensional observations forget}\label{sec:information-loss}
The examples in this section are obstruction patterns: they identify occurrence information that binary reachability, support data, or elementwise composition may forget.  Their relevance in an application depends on which distinctions its semantics is intended to retain, independently of the graph syntax used to present them.

\begin{proposition}[Unary ancestry is not separating]
Any invariant of finite multi-input presentations that factors through a binary source--target reachability relation fails to distinguish a conjunctive rule requiring both $a,b$ from two alternative rules requiring $a$ or $b$ separately.
\end{proposition}
\begin{proof}
Both presentations have paths $a\leadsto t$ and $b\leadsto t$.  Their minimal supports are respectively $\{\{a,b\}\}$ and $\{\{a\},\{b\}\}$, so the presentations differ although their unary shadows agree.
\end{proof}

\begin{proposition}[Support semantics is not sharing-sensitive]
An invariant determined only by minimal external source supports at designated outputs cannot distinguish a shared intermediate occurrence from two isomorphic but disjoint copies when the external support families agree.
\end{proposition}
\begin{proof}
A shared derivation of $x$ from $a$ used by outputs $u,v$ and two copied derivations $x_1,x_2$ from the same source both give support $\{a\}$ for each output, while their incidence structures are nonisomorphic.
\end{proof}

\begin{proposition}[Overlap data is necessary for separating binary composition]
Suppose the semantic value of a binary composite is $x\otimes y$ and depends only on the separate element-level values of its two inputs.  If a subpresentation $R$ and an isomorphic disjoint copy have the same value $r$, then the semantics cannot distinguish using one shared occurrence of $R$ twice from using two disjoint copies: both composites receive $r\otimes r$.
\end{proposition}
\begin{proof}
In the shared presentation both inputs have value $r$, and in the copied presentation the two separate inputs also have value $r$.  A composite semantics of the stated form therefore assigns $r\otimes r$ to both.
\end{proof}

\begin{remark}[Quantitative consequence]
A scalar error bound attached only to external inputs may be sound but cannot in general attribute error to a shared occurrence or determine how many independent local perturbations were introduced.  When such attribution is required, error extraction should be performed on the realization level before passing to an extensional cost or coefficient profile.
\end{remark}
\section{Flattening and reification-sensitive behavior}

\subsection{Hierarchical presentations}
A hierarchical presentation may contain a box naming or enclosing a subpresentation. A flattening operation erases box boundaries and returns the underlying incidence structure.

\begin{definition}[Code-sensitive context]
A context is code-sensitive if it can distinguish a presentation supplied as one quoted unit from one assembled from separately quoted components, even when both have isomorphic flattenings.
\end{definition}

\begin{example}[One box versus two boxes]
Let $R$ be a two-step chain $a\to b\to c$. Let $H_1$ contain one quoted box enclosing the whole chain. Let $H_2$ contain two quoted boxes, one enclosing $a\to b$ and one enclosing $b\to c$, composed externally. Their flattened incidence graphs may be isomorphic, while a context testing the number or boundaries of quoted units distinguishes them.
\end{example}

\begin{theorem}[Flattening cannot be fully abstract for code-sensitive contexts]
Let $\operatorname{Flat}:\Pres_{\mathrm{hier}}\to\Pres_{\mathrm{flat}}$ erase hierarchical boundaries. Suppose there are $P,Q$ with $\operatorname{Flat}(P)\iso\operatorname{Flat}(Q)$ and a context $C\blank$ with different observations on $C[P]$ and $C[Q]$. Then no isomorphism-invariant semantics factoring through $\operatorname{Flat}$ is fully abstract for a context language containing $C\blank$.
\end{theorem}
\begin{proof}
If $S=S_0\circ\operatorname{Flat}$ and $S_0$ is invariant under isomorphism, then $S(P)\iso S(Q)$. Full abstraction would imply contextual equivalence, contradicting the chosen observation.
\end{proof}

\begin{corollary}
A flat hypergraph semantics can be fully abstract only for contexts insensitive to the erased boundaries, or after imposing equations making those boundaries observationally irrelevant.
\end{corollary}
\begin{proof}
If an available context distinguishes presentations with isomorphic flattenings, the preceding theorem rules out full abstraction.  Thus full abstraction requires either the absence of such contexts or equations that make every such distinction observationally irrelevant.
\end{proof}
\section{Multiple realizations and deterministic selection}

\subsection{The elementary obstruction}
\begin{lemma}[No left inverse for a noninjective behavior map]
Let $q:\Pres\to\Beh$ be a function. If distinct $P,Q$ satisfy $q(P)=q(Q)$, then no function $u:\Beh\to\Pres$ satisfies $u\circ q=\mathrm{id}_{\Pres}$.
\end{lemma}
\begin{proof}
Such a function would give $P=u(q(P))=u(q(Q))=Q$.
\end{proof}

\begin{corollary}
When one behavior has several realizations, no single-valued behavior operation can recover every realization. A representative selection satisfying $q\circ u=\mathrm{id}_{\Beh}$ may exist, but it chooses one realization and does not invert $q$ on all presentations. The full fiber must be retained as a set, category, groupoid, space, or other multi-realization object if all realizations matter.
\end{corollary}
\begin{proof}
The preceding lemma rules out a left inverse of $q$.  A right inverse, when one is chosen, selects only one element in each nonempty fiber.
\end{proof}

\subsection{Partial context action on a quotient}
\begin{proposition}[Congruence and domain-saturation criterion]\label{thm:partial-congruence}
Let $\sim$ be an equivalence relation on $\Pres$, and let a context $C$ define a partial operation with domain $\operatorname{Dom}(C)\subseteq\Pres$. The rule
\[
C([P]):=[C[P]]
\]
defines a partial operation on $\Pres/{\sim}$ iff:
\begin{enumerate}
\item $\operatorname{Dom}(C)$ is $\sim$-saturated: $P\sim Q$ implies $P\in\operatorname{Dom}(C)$ iff $Q\in\operatorname{Dom}(C)$;
\item whenever $P\sim Q$ and both are in the domain, $C[P]\sim C[Q]$.
\end{enumerate}
\end{proposition}
\begin{proof}
A quotient-level partial operation must have a well-defined domain of equivalence classes, which requires saturation. On that domain, representative independence is the congruence condition.
\end{proof}
\section{One-sorted extensionality and faithful quotation}\label{sec:one-sorted}

\subsection{The reification trilemma}
\begin{definition}[Extensional equality]
An equivalence relation $\sim$ on presentations is extensional if it identifies at least one pair of distinct presentations having the same behavior.
\end{definition}

\begin{definition}[Faithfully inspectable quotation]
A quotation operation $q:\Pres\to\Pres$ is faithfully inspectable if there is a closing inspection context $I\blank$ such that $I[q(P)]$ is a closed observation for every $P\in\Pres$ and, for all distinct $P_0,P_1$,
\[
\operatorname{Obs}(I[q(P_0)])\ne\operatorname{Obs}(I[q(P_1)]).
\]
\end{definition}

\begin{theorem}[One-sorted reification trilemma]
The following three conditions cannot hold simultaneously on one presentation sort:
\begin{enumerate}[label=(\alph*)]
\item $\sim$ is extensional and observationally sound on closed observations: $X\sim Y$ implies $\operatorname{Obs}(X)=\operatorname{Obs}(Y)$ whenever both are closed observations;
\item $\sim$ is a congruence for quotation and all inspection contexts, including domain saturation as in \cref{thm:partial-congruence};
\item quotation is faithfully inspectable.
\end{enumerate}
\end{theorem}
\begin{proof}
Choose distinct $P_0\sim P_1$. Congruence gives $q(P_0)\sim q(P_1)$ and then $I[q(P_0)]\sim I[q(P_1)]$. Observational soundness forces equal observations, contradicting faithful inspectability.
\end{proof}

\begin{corollary}[Necessary design choices]
A system retaining extensional proof equality and inspectable proof codes must weaken at least one of single-sortedness, substitutivity of extensional equality in intensional contexts, faithful inspection, or observational soundness. Two possible remedies are to separate behavior and code sorts or to stratify quotation by levels.
\end{corollary}
\begin{proof}
This is the contrapositive of the incompatibility established in the preceding theorem.
\end{proof}
\section{Interaction systems and same-level quotation}\label{sec:interaction}

\begin{definition}
Fix a set $K$ with distinct elements $0,1$. An interaction system over $K$ is a triple
\[
X=(R_X,C_X,\langle-,-\rangle_X),
\]
where $R_X$ is a set of realizations, $C_X$ a set of tests, and $\langle-,-\rangle_X:R_X\times C_X\to K$ an interaction map.
\end{definition}

\begin{definition}
For $r,s\in R_X$, write $r\equiv_X s$ iff
\[
\langle r,c\rangle_X=\langle s,c\rangle_X\qquad(c\in C_X).
\]
\end{definition}

\begin{definition}[Same-level syntax-separating quotation]
Such a scheme consists of maps $q:R_X\to R_X$ and $\delta:R_X\to C_X$ satisfying:
\begin{align*}
\textnormal{(Q1)}&\quad r\equiv_X s\Longrightarrow q(r)\equiv_X q(s),\\
\textnormal{(Q2)}&\quad \langle q(r),\delta(s)\rangle_X=1\Longleftrightarrow r=s.
\end{align*}
\end{definition}

\begin{theorem}[Same-level quotation collapse]
If an interaction system admits a same-level syntax-separating quotation scheme, then $\equiv_X$ is equality.
\end{theorem}
\begin{proof}
If $r\equiv_X s$, then $q(r)\equiv_X q(s)$ by (Q1). Evaluating at $\delta(r)$ gives
\[
1=\langle q(r),\delta(r)\rangle_X=\langle q(s),\delta(r)\rangle_X,
\]
and (Q2) implies $s=r$.
\end{proof}

\begin{corollary}
If distinct realizations are extensionally equivalent, syntax-separating quotation cannot both live at the same interaction level and descend to the extensional quotient.
\end{corollary}
\begin{proof}
Descent gives (Q1), while syntax separation gives (Q2); the preceding theorem would then make extensional equivalence equal to literal equality.
\end{proof}
\section{The categorical reification obstruction}\label{sec:yoneda-obstruction}
For the Yoneda embedding, pointwise Kan extensions, and free cocompletion by presheaves, see \cite{maclane,maclane-moerdijk}.
Let $\mathcal E$ be a small category of proof realizations, $\mathcal B$ a category of behaviors, and $p:\mathcal E\to\mathcal B$ a functor. Write
\[
y_{\mathcal E}:\mathcal E\to\PSh(\mathcal E):=[\mathcal E^{\mathrm{op}},\mathbf{Set}]
\]
for Yoneda.

\begin{definition}[Representable behavior-code factorization]
A factorization consists of a functor $C:\mathcal B\to\PSh(\mathcal E)$ and a natural isomorphism
\[
\alpha:y_{\mathcal E}\xRightarrow{\sim}Cp.
\]
Thus $\alpha_e$ identifies $C(p(e))$ with the representable $y_{\mathcal E}(e)$ for every realization $e$.
\end{definition}

\begin{theorem}[Yoneda reification obstruction]
If $p$ admits a representable behavior-code factorization, then:
\begin{enumerate}[label=(\roman*)]
\item $p$ is faithful;
\item if $p(e)\iso p(e')$ in $\mathcal B$, then $e\iso e'$ in $\mathcal E$.
\end{enumerate}
In particular, two nonisomorphic proofs cannot have the same behavior.
\end{theorem}
\begin{proof}
If $p(f)=p(g)$ for $f,g:e\to e'$, naturality gives
\[
\alpha_{e'}y(f)=Cp(f)\alpha_e=Cp(g)\alpha_e=\alpha_{e'}y(g),
\]
so $y(f)=y(g)$ and $f=g$. If $u:p(e)\iso p(e')$, then $C(u)$ is an isomorphism and hence $y(e)\iso y(e')$; Yoneda implies $e\iso e'$.
\end{proof}

\begin{corollary}
If some $b\in\mathcal B$ has nonisomorphic realizations $e,e'$ with $p(e)=p(e')=b$, there is no functor $C:\mathcal B\to\PSh(\mathcal E)$ and natural isomorphism $y_{\mathcal E}\cong Cp$.
\end{corollary}
\begin{proof}
Such a factorization would imply $e\iso e'$ by the preceding theorem.
\end{proof}
\section{Behaviorwise aggregation and its category of elements}
Assume $\mathcal E$ is a small groupoid, $\mathcal B$ is a set regarded as a discrete category, and $p:\mathcal E\to\mathcal B$. Let $\mathcal E_b$ be the fiber groupoid and define
\[
\mathcal A:=\Lan_p(y_{\mathcal E}):\mathcal B\to\PSh(\mathcal E).
\]

\begin{theorem}[Behaviorwise aggregate code]\label{thm:aggregate}
For $b\in\mathcal B$ and $x\in\mathcal E$,
\[
\mathcal A(b)(x)\cong
\begin{cases}
\{*\},&p(x)=b,\\
\varnothing,&p(x)\ne b.
\end{cases}
\]
Thus $\mathcal A(b)$ is the characteristic presheaf of the full fiber $\mathcal E_b$. It is generally nonrepresentable and does not select an individual realization. Nevertheless, its category of elements is canonically isomorphic to the fiber groupoid:
\[
\int_{\mathcal E}\mathcal A(b)\cong\mathcal E_b.
\]
Hence the aggregate is one presheaf whose category of elements recovers the objects and automorphisms of the fiber.
\end{theorem}
\begin{proof}
Pointwise left Kan extension gives
\[
\mathcal A(b)(x)=\colim_{e\in\mathcal E_b}\Hom_{\mathcal E}(x,e).
\]
If $p(x)\ne b$, the set is empty because $\mathcal B$ is discrete. If $p(x)=b$, the identity of $x$ supplies a point. For $f:x\to e$ and $g:x\to e'$, the arrow $h=gf^{-1}:e\to e'$ in $\mathcal E_b$ identifies the two points in the colimit, so it is a singleton.

An object of $\int_{\mathcal E}\mathcal A(b)$ is an object $x$ with $p(x)=b$ together with the unique element of $\mathcal A(b)(x)$; its morphisms are the morphisms of $\mathcal E_b$. This yields the final isomorphism.
\end{proof}

\begin{corollary}[Representability criterion]
The aggregate $\mathcal A(b)$ is representable iff the fiber $\mathcal E_b$ is contractible.
\end{corollary}
\begin{proof}
If $\mathcal A(b)\cong y(e)$, then $\Hom(x,e)$ is a singleton for every $x\in\mathcal E_b$ and empty outside the fiber, so $e$ is terminal in the groupoid and the fiber is contractible. Conversely, a terminal object represents the terminal presheaf on its fiber.
\end{proof}

\begin{remark}
The category-of-elements construction recovers the fiber groupoid from the aggregate and therefore retains its automorphism data. Space-valued codes may still be useful for homotopical refinements.
\end{remark}
\section{Presheaf completion of proof-level codes}

\begin{theorem}[Free cocompletion by presheaves]
Let $\mathcal E$ be small. Yoneda
\[
y_{\mathcal E}:\mathcal E\to\PSh(\mathcal E)
\]
is fully faithful. The presheaf category is cocomplete, and for every cocomplete category $\mathcal C$, precomposition with Yoneda induces an equivalence
\[
\operatorname{Fun}_{\mathrm{cocont}}(\PSh(\mathcal E),\mathcal C)
\simeq
\operatorname{Fun}(\mathcal E,\mathcal C).
\]
\end{theorem}
\begin{proof}
Full faithfulness is Yoneda. A functor $F:\mathcal E\to\mathcal C$ extends by left Kan extension,
\[
\overline F(X)=\colim_{(e,x)\in\int X}F(e),
\]
which preserves colimits and satisfies $\overline F y\cong F$. Conversely, every presheaf is the colimit of its category of elements, so a cocontinuous functor is determined by its values on representables.
\end{proof}

\begin{proposition}[Projective probes]
Each representable $y(e)$ is projective with respect to pointwise surjective natural transformations, and the representables form a strong generating family of $\PSh(\mathcal E)$.
\end{proposition}
\begin{proof}
Yoneda identifies $\Hom(y(e),X)\to\Hom(y(e),Y)$ with $X(e)\to Y(e)$. If $f,g:X\to Y$ differ, they differ on some $x\in X(e)$, corresponding to a map $y(e)\to X$ that separates them.  If $\Hom(y(e),f)$ is bijective for every $e$, then every component $f_e$ is bijective, so $f$ is an isomorphism; thus the family is strong.
\end{proof}
\section{Fiberwise presheaf completion}
Let $\mathcal B$ be a category and let $F:\mathcal B^{\mathrm{op}}\to\mathbf{Cat}$ be a pseudofunctor with small fibers, or work in a fixed universe enlargement. Its Grothendieck construction is a fibration over $\mathcal B$ with fiber $F(b)$.

\begin{construction}[Fiberwise code completion]
Define
\[
\widehat F(b):=\PSh(F(b)).
\]
For $u:b'\to b$, reindexing $F(u):F(b)\to F(b')$ induces
\[
\widehat F(u):=F(u)_!:=\Lan_{F(u)^{\mathrm{op}}}:\PSh(F(b))\to\PSh(F(b')).
\]
Canonical comparison isomorphisms for iterated left Kan extensions make $\widehat F$ a pseudofunctor. Co-Yoneda gives a pseudonatural transformation $y_F:F\Rightarrow\widehat F$ with
\[
F(u)_!y_{F(b)}(e)\cong y_{F(b')}(F(u)e).
\]
\end{construction}

Let $\mathbf{Cocont}$ denote the 2-category of cocomplete categories, cocontinuous functors, and natural transformations.

\begin{theorem}[Fiberwise free cocompletion]\label{thm:fiberwise}
For a pseudofunctor $G:\mathcal B^{\mathrm{op}}\to\mathbf{Cocont}$, restriction along $y_F$ induces an equivalence between cocontinuous pseudonatural transformations $\widehat F\Rightarrow G$ and arbitrary pseudonatural transformations $F\Rightarrow UG$, including the corresponding modification categories.
\end{theorem}
\begin{proof}
Extend each component $F(b)\to G(b)$ by the free cocompletion theorem. For $u:b'\to b$, compare the cocontinuous functors
\[
G(u)\bar\alpha_b,\qquad \bar\alpha_{b'}F(u)_!: \PSh(F(b))\to G(b').
\]
They agree on representables by pseudonaturality of $\alpha$ and co-Yoneda. Density of representables extends the comparison uniquely to all presheaves and forces the coherence laws. Restriction and extension are inverse up to equivalence.
\end{proof}

\begin{corollary}[Representable quotation and transport]
The Grothendieck construction of $\widehat F$ is a fibration of cocomplete code categories over $\mathcal B$, and $y_F$ induces a cartesian functor from the original proof fibration. Yoneda represents each proof separately in its fiber, and reindexing carries representables to representables.
\end{corollary}
\begin{proof}
The Grothendieck construction of a pseudofunctor is a fibration, and the pseudonaturality isomorphism
$F(u)_!y(e)\cong y(F(u)e)$ shows that $y_F$ preserves the chosen cartesian lifts.
\end{proof}

\begin{corollary}[Representable quotation, aggregate codes, and fiberwise completion]\label{thm:three-levels}
Let $p:\mathcal E\to\mathcal B$ be a functor from a small groupoid to a discrete behavior set, with some nonempty, noncontractible fiber.
\begin{enumerate}[label=(\alph*)]
\item Representable quotation $y_{\mathcal E}$ does not factor through $\mathcal B$.
\item The behaviorwise aggregate $\Lan_p(y_{\mathcal E})(b)$ is the characteristic presheaf of $\mathcal E_b$, and its category of elements is isomorphic to $\mathcal E_b$.
\item Fiberwise presheaf completion embeds each fiber fully faithfully by Yoneda and freely adjoins all presheaves, with cocontinuous reindexing.
\end{enumerate}
Both the representables and the terminal presheaf occur in the fiberwise completion.
\end{corollary}
\begin{proof}
Part (a) is the Yoneda obstruction, part (b) is \cref{thm:aggregate}, and part (c) is \cref{thm:fiberwise}. The final assertion follows because $\PSh(\mathcal E_b)$ contains both its representables and its terminal presheaf.
\end{proof}
\section{Proof fibers and canonical proofs}

\begin{definition}
A proof realization family over $\mathcal B$ is a functor $p:\mathcal E\to\mathcal B$ whose strict fibers $\mathcal E_b$ are groupoids. Objects of $\mathcal E_b$ are proofs of $b$ and vertical morphisms are proof isomorphisms.
\end{definition}

\begin{definition}
A proof $e\in\mathcal E_b$ is canonical in the fiber if it is terminal in $\mathcal E_b$.
\end{definition}

\begin{proposition}[Canonical proof criterion]
For a groupoid $\mathcal G$, the following are equivalent:
\begin{enumerate}[label=(\roman*)]
\item $\mathcal G$ has a terminal object;
\item $\mathcal G$ is nonempty and has a unique morphism between every two objects;
\item $\mathcal G\to\mathbf 1$ is an equivalence.
\end{enumerate}
\end{proposition}
\begin{proof}
If $t$ is terminal, every object has a unique arrow to $t$, hence by invertibility a unique arrow from $t$. Any arrow $x\to y$ is forced to be $x\to t\to y$.  Conversely, under condition~(ii), choose an object $t$; the unique arrows $x\to t$ make it terminal.  A functor from $\mathcal G$ to the terminal groupoid is an equivalence exactly when it is essentially surjective and fully faithful, which here is condition~(ii).
\end{proof}

\begin{corollary}
A behavior has a canonical proof if and only if its proof groupoid is contractible. Connectedness alone gives uniqueness only up to nonunique isomorphism; nontrivial automorphisms obstruct canonicity.
\end{corollary}
\begin{proof}
Apply the equivalence of the three conditions in the preceding proposition to the proof fiber.
\end{proof}
The invariants $\pi_0(\mathcal E_b)$ and $\Aut_{\mathcal E_b}(e)$ distinguish proof-identity classes from symmetries of one proof.
\section{Transport as a fibration}

\begin{definition}
A functor $p:\mathcal E\to\mathcal B$ is a Grothendieck fibration if for every $u:b'\to p(e)$ there is a cartesian arrow $\bar u:e'\to e$ over $u$.
\end{definition}
A cartesian lift is a proof transported along $u$. It is unique up to unique vertical isomorphism, so transport is not a literal function until a cleavage is chosen.

\begin{proposition}[Uniqueness of transport]
If $\bar u:e_1\to e$ and $\bar u':e_2\to e$ are cartesian lifts of the same base arrow, there is a unique vertical isomorphism $\theta:e_1\to e_2$ with $\bar u'\theta=\bar u$.
\end{proposition}
\begin{proof}
Cartesianness of $\bar u'$ applied to $\bar u$ gives a unique vertical $\theta$; reversing the roles gives $\psi$. Uniqueness in the universal property yields $\psi\theta=1$ and $\theta\psi=1$.
\end{proof}
\section{Monodromy and coherent proof families}
For fibrations and cleavages, see \cite{streicher,maclane}.
Assume $\mathcal B$ is a connected groupoid and $p:\mathcal E\to\mathcal B$ a cloven fibration in groupoids. A cleavage determines a pseudofunctor $F:\mathcal B^{\mathrm{op}}\to\mathbf{Gpd}$. A loop $\gamma:b\to b$ induces an autoequivalence $\gamma^*:\mathcal E_b\to\mathcal E_b$.

\begin{definition}[Homotopy fixed point]
A homotopy fixed point is an object $e\in\mathcal E_b$ with isomorphisms $\theta_\gamma:\gamma^*e\iso e$ for all loops $\gamma$, satisfying unit and multiplication coherence.
\end{definition}

\begin{theorem}[Sections as homotopy fixed points]
The groupoid of cartesian sections of $p$ is equivalent to the homotopy fixed-point groupoid of the monodromy action on $\mathcal E_b$.
\end{theorem}
\begin{proof}
A cartesian section supplies $e=s(b)$ and coherent isomorphisms $\gamma^*e\iso e$. Conversely, choose for each $c$ an arrow $\tau_c:c\to b$ with $\tau_b=1$. Given $(e,\theta)$, set $s(c)=\tau_c^*e$. For $u:c\to d$, the loop $\tau_du\tau_c^{-1}$ and $\theta$ produce the required cartesian comparison. Coherence gives functoriality, and different choices of $\tau_c$ yield isomorphic sections.
\end{proof}

\begin{corollary}[Discrete-fiber obstruction]
If $\mathcal E_b$ is a discrete set $P$ and $G=\Aut_{\mathcal B}(b)$ acts on $P$, a cartesian section exists iff $P$ has a $G$-fixed point.
\end{corollary}
\begin{proof}
For a discrete fiber, a homotopy fixed point is precisely an element fixed by every member of $G$; now apply the preceding theorem.
\end{proof}

\begin{example}[A locally nonempty family without a coherent proof]
Let $\mathcal B=BC_2$ and let the fiber be the discrete set $\{p_0,p_1\}$, swapped by the nontrivial group element. Every fiber is nonempty, but no fixed point and hence no cartesian section exists. Pulling back along the universal principal $C_2$-cover, equivalently along the trivial subgroup, removes the monodromy and yields two sections.
\end{example}
\section{Local proof fragments and descent}
For stacks and descent, see \cite{maclane-moerdijk,coquand-mannaa-ruch}.
Let $(\mathcal B,J)$ be a site presented by a pretopology and admitting the fiber products needed to form double and triple overlaps, and let $p:\mathcal E\to\mathcal B$ be fibered in groupoids. For a cover $\{U_i\to U\}$, the descent groupoid consists of local objects $e_i\in\mathcal E_{U_i}$ and transition isomorphisms on $U_i\times_UU_j$ satisfying the cocycle equation on triple overlaps.

\begin{definition}
The fibration $p$ is a stack if for every cover the restriction functor
\[
\mathcal E_U\to\operatorname{Desc}(\{U_i\to U\},\mathcal E)
\]
is an equivalence of groupoids.
\end{definition}

\begin{proposition}[Local-to-global proof existence]
For a proof stack, a global proof over $U$ exists if and only if there is a descent datum on some cover of $U$.
\end{proposition}
\begin{proof}
If a global object exists, its restrictions and the canonical overlap comparisons form a descent datum.  Conversely, the stack condition says that the restriction functor to the descent groupoid is essentially surjective, so every descent datum is effective and is represented by a global object.  Full faithfulness gives uniqueness up to a unique compatible isomorphism once the datum is fixed.
\end{proof}

\begin{example}[Cocycle obstruction on triple overlap]
Suppose local fragments are isomorphic to a proof $P$ with automorphism group $G$. On a cover with nonempty triple overlap, choose transition automorphisms $g_{ij}$. They define a descent datum only if
\[
g_{ik}=g_{jk}g_{ij}
\]
on each triple overlap, with the convention determined by pullback order. Violation of this equation obstructs gluing because there is no descent datum. A nontrivial but valid $1$-cocycle need not be an obstruction: in a stack it may glue to a globally twisted object rather than to a chosen trivialization.
\end{example}
\section{Occurrence networks and minimal joint support}

\begin{definition}[Finite acyclic occurrence network]
Such a network $P$ consists of finite sets $V_P$ of fact occurrences and $E_P$ of rule occurrences, ordered input lists $\operatorname{in}(e)$ and output lists $\operatorname{out}(e)$ in $V_P$, and a set $S_P\subseteq V_P$ of source occurrences, subject to:
\begin{enumerate}
\item the directed incidence relation is acyclic;
\item no source occurrence lies in the output of a rule;
\item every nonsource occurrence has at least one producing rule.
\end{enumerate}
\end{definition}
For a finite set $S$, let $\operatorname{Ant}(S)$ be the set of inclusion-antichains in $\mathcal P(S)$.

\begin{definition}[Minimal joint-support antichain]
Define $\Sigma_P(v)\in\operatorname{Ant}(S_P)$ recursively along a topological order. For a source $s$,
\[
\Sigma_P(s)=\{\{s\}\}.
\]
For a nonsource $v$,
\[
\Sigma_P(v)=\minsub\bigcup_{\substack{e\in E_P\\v\in\operatorname{out}(e)}}
\left\{A_1\cup\cdots\cup A_m:
\operatorname{in}(e)=(v_1,\dots,v_m),\ A_i\in\Sigma_P(v_i)\right\}.
\]
For a nullary rule ($m=0$), the empty family has one tuple of choices and its union is $\varnothing$, so the braced family is $\{\varnothing\}$.
\end{definition}

\begin{proposition}[Support correctness]
For every occurrence $v$, $\Sigma_P(v)$ is the antichain of inclusion-minimal source sets from which $v$ can be generated inside $P$.
\end{proposition}
\begin{proof}
Induct over a topological order. The source case uses the assumption that sources have no producers. Every derivation of a nonsource ends with a producing rule, and by induction its input subderivations contain supports $A_i$ as in the formula. Conversely, choosing such supports and composing their subderivations with the producing rule derives $v$. Minimization removes the nonminimal source sets.
\end{proof}

\begin{proposition}[Compositional substitution law]
Suppose a source $s$ of $P$ is replaced by an acyclic occurrence network $Q$ with distinguished output $q$.  Take the occurrences of $Q$ to be disjoint from those of $P$ before identifying $q$ with the interface occurrence formerly occupied by $s$.  In the substituted network, replace every occurrence of $s$ in a support of $P$ by a member of $\Sigma_Q(q)$, distribute unions, and remove nonminimal members.  The result is the support antichain of the substituted network.
\end{proposition}
\begin{proof}
Induct along a topological order of $P$.  At a source distinct from $s$ the support is unchanged, while at the substituted source the possible minimal supports are $\Sigma_Q(q)$.  For a nonsource occurrence, the recursive support formula forms unions of supports for the inputs of each producing rule.  Substitution commutes with these finite unions and with the union over producing rules.  Applying $\min_{\subseteq}$ at each stage removes the supports that cease to be minimal and yields the support antichain of the substituted network.
\end{proof}

\begin{example}[Support does not determine sharing]
Let $P_{\mathrm{sh}}$ contain a source $a$, one derived occurrence $x$, and two rules using $x$ to produce $y,z$. Let $P_{\mathrm{dup}}$ contain two independently produced occurrences $x_1,x_2$ used separately. Then
\[
\Sigma_{P_{\mathrm{sh}}}(y)=\Sigma_{P_{\mathrm{dup}}}(y)=\{\{a\}\},\qquad
\Sigma_{P_{\mathrm{sh}}}(z)=\Sigma_{P_{\mathrm{dup}}}(z)=\{\{a\}\},
\]
but the occurrence structures differ.
\end{example}
\section{Proof realization stacks}

\begin{definition}
A proof realization stack on a site $(\mathcal B,J)$ is a stack in groupoids $p:\mathcal E\to\mathcal B$ whose objects are proof realizations and whose vertical arrows are proof isomorphisms. A structured proof realization stack is additionally equipped with a cartesian morphism of stacks to a stack of occurrence-network structures, so that reindexing and descent isomorphisms preserve the specified occurrence data.
\end{definition}

\begin{proposition}[Reification, fibrations, and descent]
Let $p:\mathcal E\to\mathcal B$ be a category of proof realizations over behaviors. Assume:
\begin{enumerate}[label=(R\arabic*)]
\item some fiber contains two nonisomorphic realizations;
\item representable proof quotation is the Yoneda embedding $y_{\mathcal E}$, and a proposed behavior-level representable code would be a factorization $y_{\mathcal E}\cong Cp$;
\item for every base arrow and proof over its codomain, a transport is specified with the cartesian universal property, and fibers are groupoids;
\item for every cover, restriction to local proofs is an equivalence with the descent groupoid.
\end{enumerate}
Then the proposed behavior-level representable factorization in (R2) does not exist; (R3) makes $p$ a category fibered in groupoids; and (R4) makes it a stack. Behaviorwise presheaf-valued aggregate codes may still exist and do not contradict this conclusion.
\end{proposition}
\begin{proof}
(R1)--(R2) contradict the Yoneda reification obstruction. The cartesian lifting property in (R3) is the definition of a fibration, and groupoid fibers make it a fibration in groupoids. (R4) is the stack condition. The final sentence follows from \cref{thm:aggregate,thm:three-levels}.
\end{proof}

\begin{remark}
The proposition isolates categorical conditions for the stated reification, transport, and descent conclusions; it does not characterize sound proof theories. Correctness, normalization, and proof identity remain logic-specific structures inside the fibers.
\end{remark}
\section{Examples and boundary cases}

\subsection{A behavior with two unrelated proofs}
Let the behavior base have one object $b$ and let $\mathcal E_b$ be the discrete groupoid on $p,q$. There are two proof-identity classes and no canonical proof. Representable quotation cannot descend to $b$. The behaviorwise aggregate is the terminal presheaf on the discrete fiber; it does not select $p$ or $q$, but its category of elements is the two-object fiber itself.

\subsection{One proof up to isomorphism with internal symmetry}
Let $\mathcal E_b=BG$ for a nontrivial group $G$. There is one proof up to isomorphism but no canonical proof because the unique object has nontrivial automorphisms.

\subsection{Path-dependent transport}
Let the base be $BC_2$ and the discrete fiber consist of two nonisomorphic proofs exchanged by the generator. Transport exists along every path, but no proof is invariant under the loop; the obstruction is a failure of a homotopy fixed point.

\subsection{No quotation}
If syntax-sensitive quotation is absent, the specific obstruction in the same-level quotation theorem does not arise, because the hypotheses of that theorem are not satisfied.  If partial contexts are to descend to an extensional quotient, their domains and actions must satisfy the saturation and congruence conditions of \cref{thm:partial-congruence}.

\compactpart{III}{Error Extraction and Guarded Quantitative Semantics}

\section{Ordered error algebras and compositional upper bounds}\label{sec:error-algebra}

\begin{definition}[Ordered error algebra]
An ordered error algebra is a unital ordered monoid $(V,\star,1,\le)$: multiplication is associative with unit $1$ and monotone in each variable.  A rule interpretation assigns to each generating rule occurrence $r$ a guarded transformer or certificate $\beta(r)\in V$.  For a nonempty path $w=r_k\cdots r_1$ set
\[
  \beta(w):=\beta(r_k)\star\cdots\star\beta(r_1).
\]
For the empty path $1_x$ at any object $x$, set $\beta(1_x):=1$.
The order records weakening of certificates; commutativity is not assumed because error propagation is generally order-sensitive.
\end{definition}

\begin{theorem}[Compositional path certificate]\label{thm:ordered-path-certificate}
Let $F(\mathcal R)$ be the free category on a graph of guarded rules, and let $V$ be an ordered error algebra.  The generator assignment extends uniquely to a functor
\[
  \beta:F(\mathcal R)\longrightarrow BV.
\]
Here $BV$ is the one-object category whose endomorphism monoid is $V$.
Suppose in addition that an application-specific predicate $\operatorname{Sound}(w,v)$ satisfies
\begin{enumerate}[label=(S\arabic*)]
\item $\operatorname{Sound}(1_x,1)$ for every object $x$;
\item $\operatorname{Sound}(r,\beta(r))$ for every generating rule $r$ on its guard;
\item if $\operatorname{Sound}(w_1,v_1)$ and $\operatorname{Sound}(w_2,v_2)$ for a composable pair, then $\operatorname{Sound}(w_2w_1,v_2\star v_1)$ on the composite guard.
\end{enumerate}
Then every empty path $1_x$ is sound with certificate $1$, and every nonempty path $w=r_k\cdots r_1$ is sound with certificate
\[
 \beta(w)=\beta(r_k)\star\cdots\star\beta(r_1).
\]
\end{theorem}
\begin{proof}
The unique functor is the universal extension from the free category, with associativity and the unit law supplied by the monoid axioms.  The soundness assertion follows by induction on path length: every empty path is covered by (S1), a generator by (S2), and the induction step by (S3).
\end{proof}

\begin{example}[Lipschitz--additive algebra]
Equip
\[
  V=[0,\infty)\times[0,\infty]
\]
with the product order, and define
\[
  (L_2,\eta_2)\star(L_1,\eta_1)
  =(L_2L_1,L_2\eta_1+\eta_2).
\]
Here $0\cdot\infty:=0$ and $L\cdot\infty:=\infty$ for $L>0$.  With these conventions, $\star$ is associative and monotone in each variable and has unit $(1,0)$, so $V$ is an ordered error algebra.
If an implementation $\widetilde f$ approximates $f$ with uniform error $\eta_1$ and $g$ is $L_2$-Lipschitz while $\widetilde g$ approximates $g$ with error $\eta_2$, then
\[
 \|\widetilde g\widetilde f-gf\|\le L_2\eta_1+\eta_2.
\]
Thus the familiar error recurrence is monoid multiplication rather than an isolated formula.
\end{example}

\begin{proposition}[Occurrence transport of local certificates]\label{prop:certificate-transport}
Let $c_G(a)\in V$ be a certificate assigned to each occurrence $a:A\mono G$, invariant under occurrence isomorphism.  Assume that
\[
  c_G(i b)=c_D(b)
\]
for every occurrence $b:B\mono D$ and every monic context embedding $i:D\mono G$ used in a DPO step.  If $a$ is stable under $\rho:G\Rightarrow H$, then
\[
  c_H(h\bar a)=c_G(a),
\]
where $\bar a:A\mono D$ is the unique factorization satisfying $d\bar a=a$ for the context span $G\xleftarrow dD\xrightarrow hH$.
\end{proposition}
\begin{proof}
By \cref{thm:maximal-step}, $a$ factors uniquely through the context $D$.  Both its source and target occurrences are obtained from $\bar a$ by the context embeddings $D\mono G$ and $D\mono H$.  The stated context-embedding condition and isomorphism invariance give the equality.
\end{proof}

\begin{remark}
The proposition applies only to intact residuals under the displayed context-embedding condition; deleted or internally rewritten subpresentations need not have unchanged error.
\end{remark}
\section{Partial affine realizations and propagated error}

\subsection{Partial affine rules}
Let $X\subseteq\mathbb R$. A partial map is always understood together with its actual domain.

\begin{definition}[Partial affine rule]
A partial affine rule on $X$ is a triple
\[
r=(G_r,a_r,b_r),
\]
where $G_r\subseteq X$, $a_r,b_r\in\mathbb R$, and $a_rx+b_r\in X$ for all $x\in G_r$. It denotes $f_r:G_r\to X$, $f_r(x)=a_rx+b_r$.
\end{definition}
For $r$ followed by $s$, define
\begin{equation}\label{eq:partial-affine-compose}
s\circ r=\bigl(G_r\cap f_r^{-1}(G_s),\ a_sa_r,\ a_sb_r+b_s\bigr).
\end{equation}

\begin{proposition}[Category of partial affine endomorphisms]\label{thm:partial-affine-category}
The partial affine rules on $X$, with composition \eqref{eq:partial-affine-compose} and identity $(X,1,0)$, form a one-object category $\mathbf{PAff}(X)$. The represented partial map of a composite is the composite of the represented partial maps.
\end{proposition}
\begin{proof}
For $x$ in the displayed domain,
\[
f_{s\circ r}(x)=a_s(a_rx+b_r)+b_s=f_s(f_r(x)).
\]
With a third rule $t$, both parenthesizations have domain
\[
G_r\cap f_r^{-1}(G_s)\cap(f_s\circ f_r)^{-1}(G_t)
\]
and coefficients $a_ta_sa_r$ and $a_ta_sb_r+a_tb_s+b_t$.
\end{proof}

\subsection{Interpretation of a rule category}
Let $F(\mathcal R)$ be the free category on a directed graph of rule symbols.

\begin{corollary}[Extension from generators]
An assignment of a partial affine rule to every generating arrow extends uniquely to a functor
\[
\mathcal I:F(\mathcal R)\to\mathbf{PAff}(X).
\]
For a composable word $w=r_k\cdots r_1$,
\[
\mathcal I(w)=\mathcal I(r_k)\circ\cdots\circ\mathcal I(r_1).
\]
\end{corollary}
\begin{proof}
The identity path is sent to $(X,1,0)$ and a nonempty path is sent to the indicated composite.  Associativity and the identity laws in \cref{thm:partial-affine-category} make this assignment functorial, and the universal property of the free category makes it unique.
\end{proof}

\begin{remark}
The source category $F(\mathcal R)$ retains distinct rule paths, but $\mathcal I$ need not be faithful: different paths may have the same partial-affine image.  With literal triples, equality of images means equality of the composite domains and both coefficients.  After quotienting by extensional equality, equality of images means that the represented partial maps have the same domain and agree there.
\end{remark}

\subsection{Affine majorants and propagated error}
\begin{definition}[Affine majorant]
Let $F:G\to X$ be a partial map with $X\subseteq[0,\infty)$. An affine majorant is a pair $(a,b)$ with $a,b\ge0$ such that
\[
F(x)\le ax+b\qquad(x\in G).
\]
\end{definition}

\begin{proposition}[Composition of affine majorants]
If $F$ has majorant $(a_F,b_F)$ and $H$ has majorant $(a_H,b_H)$ with nonnegative coefficients, then $H\circ F$ has majorant
\[
(a_Ha_F,a_Hb_F+b_H)
\]
on $G_F\cap F^{-1}(G_H)$.
\end{proposition}
\begin{proof}
On the composite domain, nonnegativity of $a_H$ gives
\[
H(F(x))\le a_HF(x)+b_H
\le a_H(a_Fx+b_F)+b_H.
\]
\end{proof}

\begin{proposition}[Additive error propagation]
Suppose
\[
F(x)=a_Fx+b_F+e_F(x),\quad |e_F(x)|\le\eta_F,
\]
and
\[
H(y)=a_Hy+b_H+e_H(y),\quad |e_H(y)|\le\eta_H.
\]
Then
\[
H(F(x))=a_Ha_Fx+(a_Hb_F+b_H)+e_{H\circ F}(x)
\]
with
\[
|e_{H\circ F}(x)|\le |a_H|\eta_F+\eta_H.
\]
\end{proposition}
\begin{proof}
For $x$ in the composite domain,
\[
e_{H\circ F}(x)=a_He_F(x)+e_H(F(x)).
\]
The triangle inequality gives the stated bound.
\end{proof}

\begin{corollary}[Error bound for a word]
For $w=r_k\cdots r_1$ with affine parts $(a_i,b_i)$ and additive error bounds $\eta_i$,
\[
\eta_w\le \eta_k+|a_k|\eta_{k-1}+|a_ka_{k-1}|\eta_{k-2}+\cdots+|a_k\cdots a_2|\eta_1.
\]
\end{corollary}
\begin{proof}
Iterate the two-rule recurrence in the preceding proposition along the word, starting with the first rule.
\end{proof}
\section{Ranked descent certificates}

\subsection{Decrease for an affine majorant}
\begin{definition}[Affine decrease threshold]
If $F:G\to[0,\infty)$ has affine majorant $(a,b)$ with $0\le a<1$, define
\[
\theta(F;a,b)=\frac{b}{1-a}.
\]
\end{definition}

\begin{proposition}[Soundness of the threshold]
If $x\in G$ and $x>\theta(F;a,b)$, then $F(x)<x$.
\end{proposition}
\begin{proof}
Since $x>b/(1-a)$ and $1-a>0$, one has $(1-a)x>b$, equivalently $ax+b<x$.  The majorant then gives $F(x)\le ax+b<x$.
\end{proof}

\subsection{Ranked transition systems}
\begin{definition}[Ranked deterministic transition system]
A ranked deterministic transition system is a triple $(X,T,\rho)$ with $T:X\to X$ and $\rho:X\to W$ into a well-ordered set $W$.
\end{definition}

\begin{definition}[Descending path datum]
A descending path datum is a pair $(D,k)$ with $D\subseteq X$, $k\ge1$, and
\[
\rho(T^k(x))<\rho(x)\qquad(x\in D).
\]
A family $\{(D_i,k_i)\}_{i\in I}$ is a descending cover outside $B\subseteq X$ if $X\setminus B\subseteq\bigcup_iD_i$.
\end{definition}

\begin{theorem}[Well-founded cover theorem]\label{thm:well-founded-cover}
If a ranked deterministic transition system has a descending cover outside $B$, then every orbit meets $B$.
\end{theorem}
\begin{proof}
If an orbit avoided $B$, repeatedly choose a covering datum for the current point and advance by its positive iterate length. This yields an infinite strictly descending sequence in the well-order $W$, impossible.
\end{proof}

\begin{corollary}[Strong-induction form]\label{cor:strong-induction}
Let $T:\Npos\to\Npos$. If some $N\ge2$ and a descending cover satisfy $T^{k_i}(n)<n$ on $D_i$, and the $D_i$ cover all $n\ge N$, every orbit reaches $\{1,\dots,N-1\}$.
\end{corollary}
\begin{proof}
Apply \cref{thm:well-founded-cover} with rank $\rho(n)=n$ and exceptional set $B=\{1,\dots,N-1\}$.
\end{proof}

\subsection{Finite rule systems}
Suppose primitive rule domains partition $X$, so every state determines a unique legal word of every finite length. If a word $w$ has domain $G_w$ and a contractive affine majorant, then $D_w(N)=\{x\in G_w:x\ge N\}$ is descending once $N$ exceeds its threshold.

\begin{proposition}[Refinement monotonicity]
Let $W_k$ be the legal words of length $k$. Mark a word terminal when one of its prefixes defines a descending domain above a fixed threshold $N$. Every extension of a terminal word is terminal; hence the union of terminal length-$k$ domains increases with $k$.
\end{proposition}
\begin{proof}
The witnessing descending prefix of a terminal word remains a prefix of every extension.  Since the primitive domains partition the state space, each state in a terminal length-$k$ domain has a unique legal length-$(k+1)$ extension, which is again terminal.
\end{proof}

\subsection{Relation to occurrence-sensitive error semantics}
The partial affine action and the occurrence model record different data. Presentations may have the same partial affine action while differing in sharing or quotation, and isomorphic occurrence networks may receive different numerical interpretations. Neither determines the other.

\section{Behaviorwise quantitative reflection}\label{sec:quant-reflection}

\subsection{The discrete universal property}
\begin{definition}[Behaviorwise reflection]
Let $O:E\to B$ be a function and $A:E\to[0,\infty]$ a magnitude.  Define
\[
  Q_A(b):=\inf\{A(e):O(e)=b\},
\]
with the convention $\inf\varnothing=\infty$.
\end{definition}

\begin{theorem}[Largest behavior-level lower bound]\label{thm:behavior-reflection}
The function $Q_A$ satisfies $Q_A\circ O\le A$.  If $H:B\to[0,\infty]$ satisfies $H\circ O\le A$, then $H\le Q_A$.  Hence $Q_A$ is the greatest quantitative statement depending only on behavior and valid for every realization of that behavior.
\end{theorem}
\begin{proof}
For $e\in E$, the value $A(e)$ is one member of the set whose infimum defines $Q_A(O(e))$.  Conversely, if $H(O(e))\le A(e)$ for every $e$ in the fiber over $b$, then $H(b)$ is below their infimum.
\end{proof}

\begin{corollary}[Right Kan description]
Regard $E,B$ as discrete categories and $[0,\infty]$ as the poset category with its usual order.  Then
\[
  Q_A=\Ran_O A.
\]
\end{corollary}
\begin{proof}
The order-enriched universal property of the right Kan extension is
\[
  H\le\Ran_OA\quad\Longleftrightarrow\quad HO\le A,
\]
which is the universal property in \cref{thm:behavior-reflection}.
\end{proof}

\subsection{Non-discrete behaviors and transport}
\begin{proposition}[Fibrational fiber formula]\label{thm:fibrational-reflection}
Let $O:\mathcal E\to\mathcal B$ be a Grothendieck fibration and let $A:\mathcal E\to[0,\infty]$ be a functor.  Whenever the right Kan extension exists, its pointwise value is
\[
 (\Ran_OA)(b)
 =\inf_{(u:b\to Oe)\in(b\downarrow O)}A(e)
 =\inf_{e\in\mathcal E_b}A(e).
\]
\end{proposition}
\begin{proof}
The first equality is the pointwise formula for a right Kan extension into a complete poset.  Strict-fiber objects form a subset of the comma objects, so the first infimum is at most the second.  Conversely, for $(e,u)$ choose a cartesian lift $\bar u:u^*e\to e$ with $u^*e\in\mathcal E_b$.  Functoriality gives $A(u^*e)\le A(e)$.  Hence the fiber infimum is below every comma value, proving the reverse inequality.
\end{proof}

\begin{remark}
Without the fibration hypothesis the comma category, not the strict fiber, controls the right Kan extension.  This distinction matters when a behavior morphism permits reindexing into cheaper realizations not literally lying over the chosen object.
\end{remark}

\subsection{Intensional and quantitative aggregation}
\begin{proposition}[Behaviorwise code aggregation and quantitative reflection]\label{thm:two-aggregations}
Let $p:\mathcal E\to\mathcal B$ be a functor from a small groupoid to a discrete behavior set, and let $A:\operatorname{Ob}(\mathcal E)\to[0,\infty]$ be isomorphism-invariant.  Then:
\begin{enumerate}[label=(\alph*)]
\item the category of elements of the left Kan aggregate $\Lan_p(y_{\mathcal E})(b)$ is isomorphic to the realization fiber $\mathcal E_b$;
\item the right Kan reflection $(\Ran_pA)(b)=\inf_{e\in\mathcal E_b}A(e)$ depends on the fiber only through this infimum.
\end{enumerate}
Neither construction determines the other, and both may be used simultaneously.
\end{proposition}
\begin{proof}
Part (a) is \cref{thm:aggregate}.  Part (b) is the discrete right Kan formula.  Fibers with identical infima can have different components and automorphism groups, while the same fiber can support different magnitude functions, so neither output determines the other.
\end{proof}

\section{Quotient costs, dual tests, and uncertainty}\label{sec:quotient-duality}
For quotient-space duality and weak-star compactness, see \cite{conway,folland}.

\subsection{Quotient duality}
All suprema of nonnegative quantities in this and the subsequent quantitative sections are taken in the complete lattice $[0,\infty]$; in particular, $\sup\varnothing=0$.
\begin{theorem}[Quotient-cost duality]\label{thm:quotient-duality}
Let $X$ be a complex normed space, $B$ a finite-dimensional complex normed vector space, and $T:X\onto B$ a continuous linear map.  Define
\[
  q_T(b):=\inf\{\|x\|_X:Tx=b\}.
\]
Then $q_T$ is the quotient norm induced by $T$ and
\[
  q_T(b)=\sup_{\varphi\in B^*,\,\varphi\ne0}
  \frac{|\varphi(b)|}{\|T^*\varphi\|_{X^*}}.
\]
\end{theorem}
\begin{proof}
The map $T$ induces a linear isomorphism $\widetilde T:X/\ker T\to B$, and $q_T$ is the transported quotient norm.  The dual of $X/\ker T$ is isometrically the annihilator $(\ker T)^\perp\subseteq X^*$.  Under $\widetilde T$, a functional $\varphi\in B^*$ corresponds to $T^*\varphi$, so the dual norm of $q_T$ is $\varphi\mapsto\|T^*\varphi\|$.  Finite-dimensional norm duality yields the formula.
\end{proof}

\begin{proposition}[Attainment in dual realization spaces]\label{prop:dual-attainment}
Suppose $X=Y^*$, the map $T$ is weak-star continuous, and $B$ is finite-dimensional.  Then the infimum defining $q_T(b)$ is attained.
\end{proposition}
\begin{proof}
Choose $x_n$ with $Tx_n=b$ and $\|x_n\|\downarrow q_T(b)$.  A fixed weak-star compact ball contains the tail.  A convergent subnet has limit $x$; weak-star continuity gives $Tx=b$, and weak-star lower semicontinuity of the norm gives $\|x\|\le q_T(b)$.
\end{proof}

\subsection{Restricted tests and quantitative completeness}
Let $\mathcal C\subseteq B^*$ be a chosen test family and put
\[
  q_{\mathcal C}(b):=
  \sup_{\substack{\varphi\in\mathcal C\\T^*\varphi\ne0}}
  \frac{|\varphi(b)|}{\|T^*\varphi\|}.
\]
If the indexing set is empty, the supremum is defined to be $0$.
\begin{proposition}[Restricted probes give lower bounds]
For every $b$, $q_{\mathcal C}(b)\le q_T(b)$.  Equality for all $b$ holds if and only if the normalized tests in $\mathcal C$ are norming for the quotient space $(B,q_T)$.
\end{proposition}
\begin{proof}
Each displayed ratio is bounded by the full supremum in \cref{thm:quotient-duality}.  Equality for every vector is the definition of a norming subset of the dual unit sphere.
\end{proof}

\begin{remark}[Reconstruction versus quantitative completeness]
A test family can be norming for $q_T$ while failing to reconstruct any particular $x$ in the fiber $T^{-1}(b)$.  Thus the observations can determine $q_T(b)$ even when the non-reconstruction results of \cref{sec:information-loss} still apply.
\end{remark}

\subsection{Robust uncertainty}
Let $U\subseteq B$ be nonempty, compact, and convex, and define its support function on the underlying real vector space by
\[
  \sigma_U(\varphi)=\sup_{u\in U}\operatorname{Re}\varphi(u).
\]
\begin{theorem}[Robust quotient reflection]\label{thm:robust-quotient}
For $b\in B$,
\[
 \inf_{u\in U}q_T(b+u)
 =\sup_{\|T^*\varphi\|\le1}
 \bigl(\operatorname{Re}\varphi(b)-\sigma_U(-\varphi)\bigr).
\]
\end{theorem}
\begin{proof}
By \cref{thm:quotient-duality},
\[
q_T(b+u)=\sup_{\|T^*\varphi\|\le1}\operatorname{Re}\varphi(b+u).
\]
Since $T$ is onto, $T^*$ is injective, so $\varphi\mapsto\|T^*\varphi\|$ is a norm on the finite-dimensional space $B^*$; its closed unit ball is therefore compact.  Thus the uncertainty set and the dual feasible set are compact convex sets in finite-dimensional real spaces, and the integrand is continuous affine in each variable.  Sion's minimax theorem \cite[Theorem~3.4]{sion} permits interchange of infimum and supremum.  Finally
\[
\inf_{u\in U}\operatorname{Re}\varphi(u)=-\sigma_U(-\varphi).
\]
\end{proof}

For a real number $t$, write $t_+=\max\{t,0\}$.
\begin{corollary}[Coordinatewise discs]
For $B=\mathbb C^m$ and $U=\{u:|u_j|\le\eta_j\}$,
\[
 \inf_{u\in U}q_T(b+u)
 =\sup_{c\ne0}
 \frac{\bigl(|\langle b,c\rangle|-\sum_j\eta_j|c_j|\bigr)_+}
 {\|T^*c\|}.
\]
\end{corollary}
\begin{proof}
The support function of $U$ is $\sigma_U(c)=\sum_j\eta_j|c_j|$.  Since $T$ is onto, $T^*$ is injective.  In the formula of \cref{thm:robust-quotient}, rotate a nonzero $c$ by a scalar of modulus one so that the real part of the pairing is its absolute value, and then use homogeneity to normalize $\|T^*c\|$ to one.  The feasible choice $c=0$ accounts for the positive part.
\end{proof}

\section{Guarded transfer and approximate completeness}\label{sec:guarded-transfer}

\begin{proposition}[Test-dominated transfer]\label{thm:test-transfer}
Let $X_{\mathrm{ext}}$ be an external realization class with observation $o:X_{\mathrm{ext}}\to B$ and magnitude $A_{\mathrm{ext}}$.  Let $\lambda$ be a norm on $B^*$.  Suppose that on the stated guard,
\[
  |\varphi(o(x))|\le A_{\mathrm{ext}}(x)\,\lambda(\varphi)
  \qquad(\varphi\in B^*).
\]
Then
\[
  A_{\mathrm{ext}}(x)\ge q_\lambda(o(x)),\qquad
  q_\lambda(b):=\sup_{\varphi\ne0}
  \frac{|\varphi(b)|}{\lambda(\varphi)}.
\]
If $\lambda(\varphi)=\|T^*\varphi\|$ for a quotient realization map $T:X\onto B$, then $q_\lambda=q_T$.
\end{proposition}
\begin{proof}
Divide the assumed inequality by $\lambda(\varphi)$ for $\varphi\ne0$ and take the supremum.  The final assertion is \cref{thm:quotient-duality}.
\end{proof}

\begin{definition}[Approximate quantitative completeness]
An external model is approximately complete for $q:B\to[0,\infty]$ if for every $b$ with $q(b)<\infty$ and every $\eta>0$ there is $x$ such that
\[
  o(x)=b,
  \qquad A_{\mathrm{ext}}(x)\le q(b)+\eta.
\]
\end{definition}

\begin{proposition}[Infimum formula from soundness and completeness]
If \cref{thm:test-transfer} gives $A_{\mathrm{ext}}(x)\ge q(o(x))$ and the external model is approximately complete, then
\[
  q(b)=\inf_{o(x)=b}A_{\mathrm{ext}}(x).
\]
The infimum need not be attained in the external model.
\end{proposition}
\begin{proof}
Soundness gives $q(b)\le\inf_{o(x)=b}A_{\mathrm{ext}}(x)$.  If $q(b)<\infty$, approximate completeness gives the reverse inequality after letting $\eta\downarrow0$.  If $q(b)=\infty$, soundness forces every magnitude in the fiber to be infinite; with $\inf\varnothing=\infty$, the same equality holds for an empty fiber.
\end{proof}

\begin{theorem}[Guarded quantitative realization principle]\label{thm:guarded-bounds}
For a realization--error datum, assume:
\begin{enumerate}[label=(G\arabic*)]
\item $U(r)$ is a sound compositional upper certificate constructed on the actual rule/domain guard;
\item $o=O\circ\Err$ satisfies a test-dominated transfer inequality with quotient cost $q$;
\item $o$ and $A\circ\Err$ are invariant under the realization isomorphisms being quotiented.
\end{enumerate}
Then every guarded realization satisfies
\[
  q(o(r))\le A(\Err(r))\le U(r).
\]
The lower endpoint is the greatest conclusion available from $o(r)$ alone when $q$ is the behaviorwise reflection; the upper endpoint may still depend on occurrence structure and path history.  The lower inequality and its middle term descend to the realization groupoid by (G3).  The full numerical sandwich descends to a further quotient only when $U$ is invariant under the isomorphisms used in that quotient.
\end{theorem}
\begin{proof}
The left inequality is \cref{thm:test-transfer}; the right inequality is (G1).  Optimality of the lower endpoint is \cref{thm:behavior-reflection}.  Assumption (G3) makes the lower and middle quantities well defined on the chosen realization-isomorphism classes.  No invariance of the occurrence-sensitive quantity $U$ is needed for the realization-level inequality; invariance of $U$ is required only if that endpoint is also to be viewed on a quotient.
\end{proof}

\subsection{Finite constraints and global realizations}
\begin{proposition}[Weak-star finite-intersection principle]\label{thm:weakstar-fip}
Let $X=Y^*$ and let $\ell_j:X\to\mathbb C$ $(j\in J)$ be weak-star continuous linear tests, equivalently evaluations against elements of $Y$.  Fix coefficients $(a_j)$ and $M\ge0$.  If for every finite $F\subseteq J$ there is $x_F\in X$ with
\[
  \|x_F\|\le M,\qquad \ell_j(x_F)=a_j\quad(j\in F),
\]
then there is one $x\in X$ satisfying all constraints and $\|x\|\le M$.
\end{proposition}
\begin{proof}
The radius-$M$ ball is weak-star compact.  For finite $F$, the set of points satisfying the constraints indexed by $F$ is weak-star closed and nonempty.  Finite intersections correspond to unions of finite index sets, so the family has the finite-intersection property.  Compactness gives a common point.
\end{proof}

\begin{remark}
This is a compactness theorem for linear constraints, not stack descent.  Its gluing mechanism is weak-star compactness rather than a Grothendieck topology and cocycle data.
\end{remark}

\compactpart{IV}{Compact Frequency-Orbit Error Observers}

\section{Continuous and discrete orbit compactifications}\label{sec:orbits}
For the continuous and discrete Kronecker--Weyl theorems, including frequencies with rational relations, see \cite{bailleul}; for harmonic-analysis background, see \cite{rudin,katznelson}.
Write $\mathbb T=\{z\in\mathbb C:|z|=1\}$, and normalize Haar measure on every compact group to have total mass one.

\subsection{Continuous frequency flows}
Fix distinct real frequencies $\Gamma=(\gamma_1,\dots,\gamma_m)$ and define
\[
 \Phi_\Gamma:\mathbb R\to\mathbb T^m,
 \qquad
 \Phi_\Gamma(y)=(e^{i\gamma_1y},\dots,e^{i\gamma_my}).
\]
Let $G_\Gamma=\overline{\Phi_\Gamma(\mathbb R)}$ and let $\chi_j(z)=z_j$ be the restricted coordinate characters.  We write
\[
 q_\Gamma:=q_{G_\Gamma,\chi},\qquad \lambda_\Gamma:=\lambda_{G_\Gamma,\chi}.
\]

\begin{proposition}[Continuous Kronecker--Weyl averaging]\label{prop:continuous-KW}
For $\phi\in C(G_\Gamma)$,
\[
 \lim_{Y\to\infty}\frac1Y\int_0^Y\phi(\Phi_\Gamma(y))\,dy
 =\int_{G_\Gamma}\phi\,dm_\Gamma.
\]
Moreover,
\[
 G_\Gamma^\perp=\{n\in\mathbb Z^m:n\cdot\Gamma=0\}.
\]
Thus $G_\Gamma=\mathbb T^m$ if and only if the frequencies are linearly independent over $\mathbb Q$.
\end{proposition}
\begin{proof}
For $n\in\mathbb Z^m$, the restricted character $z^n$ has orbit value $e^{i(n\cdot\Gamma)y}$.  Its time average and Haar integral are both one when $n\cdot\Gamma=0$ and both zero otherwise.  Restricted character polynomials are uniformly dense by Stone--Weierstrass.  The annihilator description and the full-torus criterion follow.
\end{proof}

\subsection{Discrete frequency orbits}
Let $\Theta=(\theta_1,\dots,\theta_m)\in(\mathbb R/2\pi\mathbb Z)^m$ have pairwise distinct phases and define
\[
 \Psi_\Theta(n)=(e^{in\theta_1},\dots,e^{in\theta_m}),\qquad
 G_\Theta=\overline{\Psi_\Theta(\mathbb Z)}.
\]
For the restricted coordinate characters on $G_\Theta$, write
\[
 q_\Theta:=q_{G_\Theta,\chi},\qquad \lambda_\Theta:=\lambda_{G_\Theta,\chi}.
\]
\begin{proposition}[Discrete Kronecker--Weyl averaging]\label{prop:discrete-KW}
For $\phi\in C(G_\Theta)$,
\[
 \lim_{N\to\infty}\frac1N\sum_{n=0}^{N-1}\phi(\Psi_\Theta(n))
 =\int_{G_\Theta}\phi\,dm_\Theta.
\]
Furthermore,
\[
 G_\Theta^\perp=\{k\in\mathbb Z^m:k\cdot\Theta\in2\pi\mathbb Z\}.
\]
Thus $G_\Theta=\mathbb T^m$ iff $1,\theta_1/(2\pi),\dots,\theta_m/(2\pi)$ are linearly independent over $\mathbb Q$.
\end{proposition}
\begin{proof}
Apply the geometric-sum average to each character $z^k$ and then argue by Stone--Weierstrass as in \cref{prop:continuous-KW}.
\end{proof}

\begin{remark}[Repeated phases]
In the present interpolation setup, phases equal modulo $2\pi$ are merged because they define the same character.  Alternatively, they may be retained only together with the corresponding compatibility constraints on their coefficients.  Counting periodic copies as independent probes duplicates information without changing the observation map.
\end{remark}

\begin{remark}[Relations among characters]
The annihilator lattice records the integer relations among the characters. Replacing $G_\Gamma$ or $G_\Theta$ by a full torus discards these relations and can change the quantitative reflection.
\end{remark}

\section{Compact-group interpolation and duality}\label{sec:compact-interpolation}
For character orthogonality and compact-group Fourier analysis, see \cite{rudin,katznelson}.
Let $G$ be a compact metrizable abelian group with Haar probability measure $m$, and let $\chi_1,\dots,\chi_m\in\widehat G$ be distinct characters.  For $c\in\mathbb C^m$ put
\[
 P_c(z)=\sum_{j=1}^m c_j\chi_j(z),
 \qquad
 \lambda_{G,\chi}(c)=\int_G|P_c(z)|\,dm(z).
\]
Use the Hermitian pairing $\langle a,c\rangle=\sum_ja_j\overline{c_j}$.

\begin{lemma}
The function $\lambda_{G,\chi}$ is a norm on $\mathbb C^m$.
\end{lemma}
\begin{proof}
Absolute homogeneity and the triangle inequality follow from the corresponding properties of the $L^1$ norm.  Distinct characters are orthonormal in $L^2(G)$.  If $P_c=0$ almost everywhere, continuity gives $P_c=0$ everywhere and Fourier orthogonality gives every $c_j=0$.
\end{proof}

\begin{definition}[Compact-group interpolation norm]
Define
\[
 q_{G,\chi}(a)=\inf\left\{\|f\|_{L^\infty(G)}:
 \int_G f\overline{\chi_j}\,dm=a_j\ (1\le j\le m)\right\}.
\]
The admissible set is nonempty because $f=\sum_ja_j\chi_j$ has the required coefficients.
\end{definition}

\begin{theorem}[Compact-group interpolation duality]\label{thm:compact-duality}
For every $a\in\mathbb C^m$,
\[
 q_{G,\chi}(a)=\sup_{c\ne0}
 \frac{|\langle a,c\rangle|}{\lambda_{G,\chi}(c)}.
\]
The infimum is attained in $L^\infty(G)$.  For every $\eta>0$ there is also $f\in C(G)$ whose coefficients are $a_1,\dots,a_m$ and
\[
 \|f\|_\infty\le q_{G,\chi}(a)+\eta.
\]
\end{theorem}
\begin{proof}
Let $T:L^\infty(G)\to\mathbb C^m$ be the coefficient map.  Orthogonality gives $T(\chi_k)=e_k$, so $T$ is onto.  For $Tf=a$,
\[
 \langle a,c\rangle=\int_G f\overline{P_c}\,dm,
\]
and the adjoint functional has norm $\|P_c\|_1=\lambda_{G,\chi}(c)$; equality is attained on the unit ball by $f=P_c/|P_c|$ off the zero set.  The quotient duality theorem gives the formula.  Attainment follows from \cref{prop:dual-attainment}, since $L^\infty=(L^1)^*$ and the coefficient map is weak-star continuous.

For continuous approximate attainment, the closed $L^\infty$ ball is the weak-star closure of its continuous part.  This follows from regularity, Lusin approximation, Tietze extension, and Hahn--Banach separation \cite{conway,folland}: otherwise the measurable phase of a separating $L^1$ function could be approximated by a continuous unit-ball function, contradicting strict separation.  Approximate a minimizer on the finitely many coefficient functionals and correct the small errors by adding a small linear combination of the characters.
\end{proof}

\begin{corollary}[Basic size comparison]\label{cor:basic-size}
For every $a\in\mathbb C^m$,
\[
 \|a\|_2\le q_{G,\chi}(a)\le\|a\|_1.
\]
\end{corollary}
\begin{proof}
For $a=0$, the zero interpolant gives $q_{G,\chi}(0)=0$.  If $a\ne0$, orthogonality gives $\lambda_{G,\chi}(a)\le\|P_a\|_2=\|a\|_2$; using $c=a$ in the dual formula gives the lower bound.  The interpolant $\sum a_j\chi_j$ gives the upper bound.
\end{proof}

\section{Continuous Abel transfer and Mellin boundary data}\label{sec:continuous-transfer}
For Mellin-transform conventions and the standard pole--asymptotic correspondence, see \cite{zagier}.  The Abel estimate in \cref{thm:continuous-abel} is used for the Mellin bounds in this section.
For measurable $F:[0,\infty)\to\mathbb C$ define
\[
 A_\infty^c(F)=\lim_{Y\to\infty}\esssup_{y\ge Y}|F(y)|.
\]

\begin{lemma}[Ces\`aro-to-Abel transfer for functions]\label{lem:cesaro-abel-cont}
If $g$ is bounded, locally integrable, and $Y^{-1}\int_0^Yg\to L$, then
\[
 \varepsilon\int_0^\infty e^{-\varepsilon y}g(y)\,dy\to L
 \qquad(\varepsilon\downarrow0).
\]
\end{lemma}
\begin{proof}
Set $G(Y)=\int_0^Yg=LY+o(Y)$.  Integration by parts gives
\[
 \varepsilon\int_0^\infty e^{-\varepsilon y}g(y)\,dy
 =\varepsilon^2\int_0^\infty e^{-\varepsilon Y}G(Y)\,dY.
\]
After $t=\varepsilon Y$, the main term is $L\int_0^\infty te^{-t}dt=L$ and the remainder tends to zero by dominated convergence.
\end{proof}

\begin{theorem}[Continuous Abel coefficient lower bound]\label{thm:continuous-abel}
Let $F\in L^1_{\mathrm{loc}}([0,\infty))$ satisfy
\[
 \int_0^\infty|F(y)|e^{-\varepsilon y}\,dy<\infty
 \qquad(\varepsilon>0).
\]
For distinct frequencies $\Gamma$, assume the limits
\[
 a_j=\lim_{\varepsilon\downarrow0}
 \varepsilon\int_0^\infty F(y)e^{-\varepsilon y}e^{-i\gamma_jy}\,dy
\]
exist.  Then
\[
 A_\infty^c(F)\ge q_\Gamma(a).
\]
Moreover, $q_\Gamma(a)$ is the infimum over functions constrained only by these coefficients, the integrability guard, and the orbit group.
\end{theorem}
\begin{proof}
For $c\in\mathbb C^m$, linearity gives convergence of the weighted pairing to $\langle a,c\rangle$.  By \cref{prop:continuous-KW,lem:cesaro-abel-cont},
\[
 \varepsilon\int_0^\infty e^{-\varepsilon y}|P_c(\Phi_\Gamma(y))|\,dy
 \to\lambda_\Gamma(c).
\]
If $M=A_\infty^c(F)<\infty$, then beyond some $Y_0$ one has $|F|\le M+\delta$ almost everywhere.  The weighted contribution of $[0,Y_0]$ tends to zero, and therefore
\[
 |\langle a,c\rangle|\le(M+\delta)\lambda_\Gamma(c).
\]
Let $\delta\downarrow0$ and use \cref{thm:compact-duality}.  For optimality, choose a continuous near-minimizer $f$ on $G_\Gamma$ and set $F(y)=f(\Phi_\Gamma(y))$; Kronecker--Weyl averaging and \cref{lem:cesaro-abel-cont} give the coefficients.
\end{proof}

Let $E:[1,\infty)\to\mathbb C$ and fix $\beta\in\mathbb R$.  Assume
\[
 \int_1^\infty|E(x)|x^{-\sigma-1}\,dx<\infty
 \qquad(\sigma>\beta),
\]
and define
\[
 M_E(s)=\int_1^\infty E(x)x^{-s-1}\,dx,
 \qquad
 A_\beta(E):=\lim_{X\to\infty}
 \operatorname*{ess\,sup}_{x\ge X}\frac{|E(x)|}{x^\beta}.
\]
Here the essential supremum is taken with respect to $dx/x$.

\begin{corollary}[Radial Mellin coefficient transfer]\label{thm:mellin-transfer}
If the finite radial limits
\[
 r_j=\lim_{\varepsilon\downarrow0}
 \varepsilon M_E(\beta+\varepsilon+i\gamma_j)
\]
exist at distinct frequencies, then
\[
 A_\beta(E)\ge q_\Gamma(r)\ge\|r\|_2.
\]
\end{corollary}
\begin{proof}
Put $F(y)=e^{-\beta y}E(e^y)$.  Change of variables gives
\[
 M_E(\beta+\varepsilon+i\gamma)
 =\int_0^\infty F(y)e^{-\varepsilon y}e^{-i\gamma y}\,dy,
\]
and $A_\beta(E)=A_\infty^c(F)$.  Apply \cref{thm:continuous-abel,cor:basic-size}.
\end{proof}

\begin{corollary}[Simple Mellin boundary poles]
If a meromorphic continuation agrees with the Mellin integral on the interior side of a neighborhood of $\rho_j=\beta+i\gamma_j$ and has a simple pole there with residue $r_j$, then the preceding lower bound holds.
\end{corollary}
\begin{proof}
Near $\rho_j$, write $M_E(s)=r_j/(s-\rho_j)+H_j(s)$.  Then $\varepsilon M_E(\rho_j+\varepsilon)\to r_j$.
\end{proof}

\begin{proposition}[Higher-order radial divergence]
If $\limsup_{\varepsilon\downarrow0}|\varepsilon M_E(\beta+\varepsilon+i\gamma)|=\infty$, then $A_\beta(E)=\infty$.  In particular, a boundary pole of order at least two, under the same interior-agreement guard, rules out $E(x)=O(x^\beta)$ in the essential-tail sense.
\end{proposition}
\begin{proof}
If $A_\beta(E)<\infty$, the one-frequency proof of \cref{thm:continuous-abel} gives
\[
 \limsup_{\varepsilon\downarrow0}|\varepsilon M_E(\beta+\varepsilon+i\gamma)|\le A_\beta(E),
\]
a contradiction.  A pole of order at least two makes the left side diverge.
\end{proof}

\begin{remark}[Analytic guard]
Meromorphic continuation alone does not identify the radial coefficients of the original error term.  The Mellin integral must remain valid for every $\operatorname{Re}s>\beta$, or a separate Abelian or Tauberian theorem must supply the same boundary values.
\end{remark}

\section{Discrete Abel transfer and generating functions}\label{sec:discrete-transfer}
For a sequence $F=(F_n)_{n\ge0}$ define
\[
 A_\infty^d(F)=\lim_{N\to\infty}\sup_{n\ge N}|F_n|.
\]

\begin{lemma}[Ces\`aro-to-Abel transfer for sequences]\label{lem:cesaro-abel-disc}
If $(g_n)$ is bounded and $N^{-1}\sum_{n=0}^{N-1}g_n\to L$, then
\[
 (1-r)\sum_{n\ge0}g_nr^n\to L\qquad(r\uparrow1).
\]
\end{lemma}
\begin{proof}
For $S_N=\sum_{n=0}^Ng_n=(N+1)L+o(N)$, summation by parts gives
\[
 (1-r)\sum_{n\ge0}g_nr^n=(1-r)^2\sum_{N\ge0}S_Nr^N.
\]
The main term is $L$ and the remainder tends to zero by the discrete analogue of the scaling argument in \cref{lem:cesaro-abel-cont}.
\end{proof}

\begin{theorem}[Discrete Abel coefficient lower bound]\label{thm:discrete-abel}
Assume $\sum_{n\ge0}|F_n|r^n<\infty$ for $0<r<1$ and, for pairwise distinct phases $\Theta$,
\[
 a_j=\lim_{r\uparrow1}(1-r)\sum_{n\ge0}F_nr^ne^{-in\theta_j}
\]
exists.  Then
\[
 A_\infty^d(F)\ge q_\Theta(a).
\]
Moreover, $q_\Theta(a)$ is the infimum over sequences constrained only by the Abel coefficients, convergence guard, and discrete orbit group.
\end{theorem}
\begin{proof}
Use the test polynomial $P_c(n)=\sum_jc_je^{in\theta_j}$.  Linearity gives the coefficient pairing; \cref{prop:discrete-KW,lem:cesaro-abel-disc} give the Abel limit of $|P_c(n)|$.  Split off finitely many initial terms and bound the tail by $A_\infty^d(F)+\delta$, then apply \cref{thm:compact-duality}.  Continuous near-minimizers restricted to $\Psi_\Theta(n)$ prove optimality.
\end{proof}

\begin{corollary}[Boundary poles of a generating function]\label{cor:gf-poles}
Let $\mathcal F(z)=\sum_{n\ge0}F_nz^n$.  Suppose that near $\zeta_j=e^{-i\theta_j}$ a meromorphic continuation agreeing with the power series inside the disc has the form
\[
 \mathcal F(z)=\frac{a_j}{1-e^{i\theta_j}z}+H_j(z)
\]
with $H_j$ holomorphic.  Then $A_\infty^d(F)\ge q_\Theta(a)$.  If $R_j$ is the residue at $\zeta_j$, then $a_j=-e^{i\theta_j}R_j$.
\end{corollary}
\begin{proof}
Along $z=re^{-i\theta_j}$, $(1-r)\mathcal F(z)\to a_j$.  The residue identity follows by differentiating $1-e^{i\theta_j}z$ at $\zeta_j$.
\end{proof}

\section{Full-torus bounds, character relations, uncertainty, and countable interpolation}\label{sec:character-relations}

\subsection{The full-torus case}
If the joint character map $G\to\mathbb T^m$ is an isomorphism of compact groups, the coordinate characters are independent Steinhaus variables.
\begin{lemma}[Steinhaus $L^1$--$L^2$ comparison]
For $S=\sum_jc_jZ_j$ with independent Steinhaus variables,
\[
 \frac1{\sqrt2}\|c\|_2\le\mathbb E|S|\le\|c\|_2.
\]
\end{lemma}
\begin{proof}
One has $\mathbb E|S|^2=\|c\|_2^2$ and
\[
 \mathbb E|S|^4=2\|c\|_2^4-\sum_j|c_j|^4\le2\|c\|_2^4.
\]
Log-convexity $\|S\|_2\le\|S\|_1^{1/3}\|S\|_4^{2/3}$ gives the lower bound; Cauchy--Schwarz gives the upper bound.
\end{proof}

\begin{corollary}[Full-torus Euclidean comparison]\label{thm:full-torus-scale}
If the joint character map
\[
  G\longrightarrow\mathbb T^m,\qquad
  z\longmapsto(\chi_1(z),\dots,\chi_m(z))
\]
is an isomorphism of compact groups, then
\[
 \|a\|_2\le q_{G,\chi}(a)\le\sqrt2\|a\|_2.
\]
\end{corollary}
\begin{proof}
The lower bound is \cref{cor:basic-size}.  The Steinhaus comparison gives $\lambda(c)\ge\|c\|_2/\sqrt2$; insert this in \cref{thm:compact-duality} and use Cauchy--Schwarz.
\end{proof}

\subsection{Integer characters on one circle}
For an integer $m\ge1$, let $G=\mathbb T$ and $\chi_j(z)=z^j$, $1\le j\le m$.  Then
\[
 \lambda(\mathbf1_m)=\frac1{2\pi}\int_{-\pi}^{\pi}
 \left|\sum_{j=1}^me^{ijt}\right|dt\le1+\log m.
\]
Indeed the Dirichlet sum is bounded by $\min(m,\pi/|t|)$ and the integral splits at $\pi/m$.
\begin{proposition}[Lower bound for integer characters]\label{thm:coupled-lower-bound}
For the integer characters on one circle,
\[
 q_{\mathbb T,(z^j)}(\mathbf1_m)\ge\frac{m}{1+\log m},
\]
which grows faster than $\|\mathbf1_m\|_2=\sqrt m$.
\end{proposition}
\begin{proof}
Use $c=\mathbf1_m$ in \cref{thm:compact-duality} and the Dirichlet estimate.
\end{proof}

\begin{remark}
Dependence alone does not guarantee improvement.  A stronger bound occurs when a coefficient direction compatible with the character relations has unusually small orbit $L^1$ cost.
\end{remark}

\subsection{Robust coefficients}
For a compact convex coefficient uncertainty set $b+U\subseteq\mathbb C^m$, \cref{thm:robust-quotient} specializes to
\[
 \inf_{u\in U}q_{G,\chi}(b+u)
 =\sup_{\lambda_{G,\chi}(c)\le1}
 \bigl(\operatorname{Re}\langle b,c\rangle-\sigma_U(-c)\bigr).
\]
The formula minimizes over coefficient uncertainty without selecting a realizing function.

\subsection{Countable character families}
Let $(\chi_j)_{j\ge1}$ be pairwise distinct characters on a compact metrizable abelian group $G$, and let $a=(a_j)$.
\begin{theorem}[Uniform finite interpolation criterion]\label{thm:countable-interpolation}
For $M\ge0$, the following are equivalent:
\begin{enumerate}[label=(\roman*)]
\item every finite coefficient subproblem has interpolation norm at most $M$;
\item there exists $f\in L^\infty(G)$ with $\|f\|_\infty\le M$ and
\[
 \int_Gf\overline{\chi_j}\,dm=a_j\qquad(j\ge1).
\]
\end{enumerate}
Consequently the least simultaneous bounded realization norm is the supremum of the finite interpolation norms.
\end{theorem}
\begin{proof}
By attainment in \cref{thm:compact-duality}, condition~(i) supplies, for each finite coefficient set, an interpolant whose norm is at most $M$.  The forward implication is then \cref{thm:weakstar-fip} applied to the weak-star compact ball of $L^\infty(G)$ and the character coefficient functionals.  Conversely, one simultaneous interpolant of norm at most $M$ restricts to an admissible interpolant for every finite coefficient set.
\end{proof}

\begin{corollary}[Finite-truncation lower bounds]
A half-line function or sequence carrying all the corresponding Abel coefficients has tail amplitude at least the supremum of the finite interpolation norms.  Divergence of that supremum forces unbounded tail amplitude.
\end{corollary}
\begin{proof}
Apply \cref{thm:continuous-abel} or \cref{thm:discrete-abel} to every finite coefficient set and then take the supremum of the resulting lower bounds.
\end{proof}

\begin{example}[Square-summable coefficients forcing unboundedness]
On $G=\mathbb T$, take $\chi_n(z)=z^n$ and
\[
 a_n=\frac1{\sqrt n\log(n+1)}\qquad(n\ge2).
\]
Then $(a_n)\in\ell^2$.  For $J_m=\{2,\dots,m\}$, testing the dual formula with all coefficients one gives numerator $\gg\sqrt m/\log m$ and denominator $O(\log m)$, so
\[
 q_{J_m}(a_{J_m})\gg\frac{\sqrt m}{\log^2m}\to\infty.
\]
Thus any guarded continuous or discrete Abel realization carrying all these coefficients is unbounded on its tails.  A continuous realization with these coefficients exists: Riesz--Fischer \cite{katznelson,rudin} supplies $f\in L^2(0,2\pi)$ with Fourier coefficients $a_n$ for $n\ge2$ and zero at the other prescribed modes.  Its periodic extension lies in $L^1$ on one period, and the geometric-series identity for its exponentially weighted periodization shows that its continuous Abel coefficients are the $a_n$.  The preceding finite-truncation lower bounds then force its tail amplitude to be infinite.  Square summability alone therefore does not imply boundedness for these dependent characters.
\end{example}

\section{A finite-field curve realization}\label{sec:finite-field}
The point-count identity and Frobenius modulus bound $|\alpha_j|=q^{1/2}$ used in this section are from \cite{weil,deligne}.
Let $C/\mathbb F_q$ be a smooth, projective, geometrically connected curve of genus $g\ge1$.  Write $N_n=\#C(\mathbb F_{q^n})$.  The reciprocal Frobenius roots $\alpha_1,\dots,\alpha_{2g}$ satisfy
\[
 N_n=q^n+1-\sum_{j=1}^{2g}\alpha_j^n,
 \qquad |\alpha_j|=q^{1/2}.
\]
Group equal normalized roots into distinct phases $e^{i\theta_1},\dots,e^{i\theta_r}$ with multiplicities $m_1,\dots,m_r$.  For $n\ge1$ set
\[
 F_n=q^{-n/2}(N_n-q^n-1)=-\sum_{\nu=1}^rm_\nu e^{in\theta_\nu}.
\]
Choose $F_0$ arbitrarily.  This finite initial choice changes neither the Abel boundary coefficients nor the tail limsup.

\begin{theorem}[Frobenius orbit bound and limsup formula]\label{thm:frobenius}
With $\Theta=(\theta_1,\dots,\theta_r)$,
\[
 \limsup_{n\to\infty}\frac{|N_n-q^n-1|}{q^{n/2}}
 \ge q_\Theta(-m),
\]
and in fact
\[
 \limsup_{n\to\infty}\frac{|N_n-q^n-1|}{q^{n/2}}
 =\max_{z\in G_\Theta}\left|\sum_{\nu=1}^rm_\nu z_\nu\right|.
\]
\end{theorem}
\begin{proof}
After the harmless choice of $F_0$, the normalized error is the restriction to the discrete orbit of the continuous trigonometric polynomial $f(z)=-\sum m_\nu z_\nu$.  Its Abel coefficient vector is $-m$, so \cref{thm:discrete-abel} gives the universal lower bound.  The closure of every positive tail $\{\Psi_\Theta(n):n\ge N\}$ equals the compact cyclic group $G_\Theta$; hence continuity of $|f|$ implies that the tail orbit values approach the maximum on $G_\Theta$.
\end{proof}

\begin{remark}
The lower bound $q_\Theta(-m)$ is optimal among sequences constrained only by these Abel coefficients and the orbit group $G_\Theta$.  The maximum in \cref{thm:frobenius} also uses the specified trigonometric-polynomial formula $F_n=-\sum_\nu m_\nu e^{in\theta_\nu}$.  Closed-point counts instead involve additional divisor corrections.
\end{remark}

\compactpart{V}{Combining Upper and Lower Bounds}

\section{Occurrence-based upper bounds and behavior-based lower bounds}\label{sec:synthesis}
The maps $p,\Err,O,A,$ and $Q$ are related by the following diagram:
\[
\begin{tikzcd}[column sep=3.2em]
\mathcal R \arrow[r,"\Err"] \arrow[d,"p"']
  & \mathcal E \arrow[r,"O"] \arrow[d,"A" description]
  & \mathcal B_{\mathrm{err}} \arrow[d,"Q"]\\
\mathcal S
  & {[0,\infty]} \arrow[r,equal]
  & {[0,\infty].}
\end{tikzcd}
\]
The upper certificate $U$ is computed on $\mathcal R$ from occurrence-sensitive paths; the lower reflection $Q$ is computed from observed error behavior.  The quotation and aggregate-code constructions are not used in defining either numerical map.

\begin{theorem}[Combined realization--error bounds]\label{thm:integrated}
Assume:
\begin{enumerate}[label=(I\arabic*)]
\item realizations form a category fibered in groupoids over specifications, with small fibers or in a fixed universe enlargement, and carry occurrence structure transported by cartesian arrows;
\item an error extractor $\Err$ is cartesian or otherwise invariant under the chosen realization isomorphisms;
\item local guarded rules carry sound certificates in an ordered error algebra, yielding $A(\Err(r))\le U(r)$;
\item the observed error behavior admits a behaviorwise reflection $Q$; in the discrete case $Q(b)=\inf_{O(e)=b}A(e)$, and in the non-discrete case $A$ is functorial and $Q=\Ran_OA$;
\item whenever $Q$ is evaluated by a test family, that family satisfies the norming or stated lower-bound hypotheses of the applicable transfer theorem.
\end{enumerate}
Then:
\begin{enumerate}[label=(\alph*)]
\item every realization satisfies $Q(O\Err(r))\le A(\Err(r))\le U(r)$;
\item the left endpoint is invariant under all realization changes invisible to $O$, whereas the right endpoint may distinguish those realizations through occurrence structure;
\item the Yoneda embeddings into the fiberwise presheaf completion keep nonisomorphic proof or program realizations distinct, whereas $Q$ need not distinguish them;
\item if the error-object model with observation $o=O$, magnitude $A_{\mathrm{ext}}=A$, and lower function $q=Q$ is approximately complete in the sense of \cref{sec:guarded-transfer}, then
\[
  Q(b)=\inf\{A(e):O(e)=b\};
\]
\item under the respective hypotheses of \cref{thm:robust-quotient,thm:weakstar-fip}, minimax treats compact-convex uncertainty and weak-star compactness treats countable observations while leaving the realization category and its occurrence data unchanged.
\end{enumerate}
\end{theorem}
\begin{proof}
In the discrete case, the left inequality in (a) is \cref{thm:behavior-reflection}.  In the non-discrete case it is the right-Kan counit $(\Ran_OA)O\le A$; under the additional fibration hypothesis it also has the strict-fiber form of \cref{thm:fibrational-reflection}.  Combine this inequality with (I3).  When the chosen tests are norming for $Q$, the same sandwich is also an instance of \cref{thm:guarded-bounds}.  Part (b) follows from the factorization of $Q$ through $O$ and from the realization-level construction of $U$.  Part (c) combines (I1) with \cref{thm:fiberwise,thm:two-aggregations}.  Part (d) is the infimum formula from soundness and approximate completeness, applied with the data stated there.  Part (e) is \cref{thm:robust-quotient,thm:weakstar-fip}.
\end{proof}

\subsection{Function errors}
Let a formula or analytic approximation produce $J(r)(x)$ for a target $\tau(s)(x)$ and set $E_r(x)=J(r)(x)-\tau(s)(x)$.  A local analytic derivation may produce $U(r)$ by Lipschitz or affine composition.  If Mellin boundary coefficients of $E_r$ are justified by the original integral, then, for the resulting residue vector $\mathbf a$,
\[
 q_\Gamma(\mathbf a)\le A_\beta(E_r)\le U(r).
\]
The lower bound says that no error function with those coefficients can decay below the reflected scale; the upper bound certifies what the chosen derivation achieves.

\subsection{Algorithmic and sequence errors}
For an algorithm returning a sequence $Y_n$ approximating targets $S_n$, put $F_n=Y_n-S_n$.  Rule-level rounding, truncation, or approximation errors propagate to an upper certificate.  Boundary poles of the guarded generating function of $F$ produce discrete coefficients and hence
\[
 q_\Theta(a)\le\limsup_{n\to\infty}|F_n|\le U(r),
\]
where $U(r)$ is understood as a certified upper bound for the tail amplitude.  Repeated phases are merged, and finite initial terms do not affect the tail reflection.

\subsection{Proof and rewrite errors}
A proof realization may carry a residual obligation, failed side condition, quantitative annotation, or approximate semantic value as its error object.  Track spans determine the greatest pathwise domain of intact occurrence transport.  Under the context-embedding hypothesis of \cref{prop:certificate-transport}, certificates on transported occurrences are unchanged.  Representable Yoneda codes keep nonisomorphic realizations distinct, whereas $Q$ records only what the chosen tests force about the residual magnitude.

\section{Conditions and limitations}\label{sec:limitations}
\begin{enumerate}
\item \textbf{Error extraction is problem-specific.}  The application must supply the error extractor, difference or residual, and magnitude function, and these data must respect its specification semantics.
\item \textbf{Guards are mathematical hypotheses.}  A partial rule must be applied on its actual domain.  A Mellin continuation or generating-function continuation must agree with the original interior representation, or an independent transfer theorem is required.
\item \textbf{Behaviorwise optimality is relative to the observation map.}  When the error and behavior categories are discrete, or when $Q=\Ran_OA$ and $O$ satisfies the fibration hypothesis of \cref{thm:fibrational-reflection},
\[
Q(b)=\inf_{O(e)=b}A(e).
\]
The same formula follows when the approximate-completeness hypothesis of \cref{sec:guarded-transfer} applies.  For a general non-discrete observation functor without one of these hypotheses, whenever the pointwise right Kan extension exists,
\[
(\Ran_OA)(b)
=\inf_{(u:b\to O(e))\in(b\downarrow O)}A(e),
\]
so comma objects, rather than only realizations in the strict fiber, determine the value.  In \cref{thm:frobenius}, the coefficient-constrained lower bound is distinguished from the exact limsup of the specified trigonometric realization.
\item \textbf{Finite simple boundary poles do not force unboundedness.}  A nonzero coefficient vector gives a positive lower bound for the normalized tail amplitude.  Radial divergence and divergence of the finite-truncation interpolation norms are two sufficient criteria for unboundedness.
\item \textbf{Quantitative completeness does not reconstruct realizations.}  Norming tests can determine a quotient cost while being unable to distinguish sharing, quotation boundaries, or proof identity.
\item \textbf{Frequency observations have limited scope.}  They are suited to asymptotic and oscillatory errors, not to every finite-time, nonlinear, probabilistic, or adversarial error.  Other observation maps require their own transfer estimates and definitions of $Q$.
\item \textbf{Compactness is not descent.}  Countable coefficient realization uses weak-star compactness; proof gluing uses stack descent.  Neither argument supplies the conclusion of the other.
\item \textbf{Upper and lower bounds may differ.}  Such a gap may result from a nonoptimal realization, a weak upper estimate, or information discarded by the observer.  Equality requires a separate optimality argument.
\end{enumerate}

\begin{theorem}[Compact-orbit lower and upper bounds]\label{thm:summary}
Let $r$ be a realization with extracted error $e=\Err(r)$ and a sound occurrence-based upper certificate $U(r)$.  Suppose either:
\begin{enumerate}[label=(\roman*)]
\item $e$ is a half-line error with guarded continuous Abel coefficients $a$ along a real frequency orbit, and its tail magnitude satisfies $A_\infty^c(e)\le A(e)$; or
\item $e$ is a sequence error with guarded discrete Abel coefficients $a$ along an integer frequency orbit, and its tail magnitude satisfies $A_\infty^d(e)\le A(e)$.
\end{enumerate}
Here $A_\infty^c(e)$ and $A_\infty^d(e)$ denote the relevant continuous and discrete tail amplitudes; in the principal applications $A(e)$ is defined to be that tail amplitude.  Let $(G,\chi)$ denote the associated orbit group and distinct character family.  Then
\[
 q_{G,\chi}(a)\le A(e)\le U(r).
\]
By \cref{thm:compact-duality}, the lower endpoint has the displayed compact-group dual formula.  If the coefficients are known only up to compact-convex uncertainty, \cref{thm:robust-quotient} gives the corresponding minimax lower endpoint.  For a compatible countable character family, \cref{thm:countable-interpolation} gives the simultaneous lower endpoint as the supremum of the finite truncations.  If the orbit group is a full torus, the finite-dimensional endpoint is comparable to $\|a\|_2$; relations among the characters can make it substantially larger.
\end{theorem}
\begin{proof}
In case (i), \cref{thm:continuous-abel} gives $q_{G,\chi}(a)\le A_\infty^c(e)\le A(e)$.  In case (ii), \cref{thm:discrete-abel} gives $q_{G,\chi}(a)\le A_\infty^d(e)\le A(e)$.  The upper inequality is the certificate hypothesis.  The uncertainty-adjusted formula, dual formula, and countable-family formula are \cref{thm:compact-duality,thm:robust-quotient,thm:countable-interpolation}.  The full-torus comparison is \cref{thm:full-torus-scale}, and \cref{thm:coupled-lower-bound} gives the larger lower bound for dependent integer characters.
\end{proof}

\newpage

\section*{Declaration}
The author used ChatGPT to generate parts of the proofs.

\appendix
\section{Index of notation}
\begin{longtable}{@{}>{\raggedright\arraybackslash}p{.19\linewidth}p{.77\linewidth}@{}}
$\mathcal S,\mathcal R$ & specification and realization categories\\
$\mathcal E$ & category of structured error objects\\
$\Err(r)$ & error object extracted from realization $r$ relative to its specification\\
$O(e)$ & observed error behavior\\
$A(e)$ & true magnitude of an error object\\
$U(r)$ & occurrence-based upper certificate\\
$Q_A(b)$ & behaviorwise quantitative reflection $\inf_{O(e)=b}A(e)$\\
$\Ran_OA$ & right Kan extension controlled in general by $(b\downarrow O)$, and by the strict fiber in the discrete or fibrational cases\\
$q_T$ & quotient cost induced by a surjective linear observation map $T$\\
$\lambda_{G,\chi}$ & $L^1$ cost of a character polynomial on a compact group\\
$q_{G,\chi}$ & least $L^\infty$ realization norm with prescribed character coefficients\\
$G_\Gamma,G_\Theta$ & continuous and discrete frequency-orbit closures\\
$A_\infty^c,A_\infty^d$ & continuous tail essential amplitude and discrete tail supremum\\
$A_\beta(E)$ & normalized Mellin tail amplitude\\
$\trk(\pi)$ & partial track span of a rewrite path\\
$\rho_\#$ & canonical transport on the stable principal ideal\\
$\mathcal E_b$ & realization fiber over behavior $b$\\
$\Sigma_P(v)$ & minimal joint-support antichain of occurrence $v$\\
\end{longtable}

\section{Bibliographic notes}\label{sec:bibliographic-notes}
For categories, topoi, fibrations, and presheaves, see \cite{maclane,maclane-moerdijk,streicher}; for stack semantics, see \cite{coquand-mannaa-ruch}.  Whole-grain occurrence models and adhesive compositional rewriting are developed in \cite{kock-petri,lack-sobocinski,behr-harmer-krivine,zanasi-free-hypergraph,bonchi-rewriting-frobenius}.  The discussions of proof identity, sharing, and intensional stratification draw on \cite{girard,strassburger,hasegawa,pfenning,hu-pientka}.  For harmonic analysis and compact-group Fourier theory, see \cite{rudin,katznelson}; for functional-analytic duality and weak-star compactness, see \cite{conway,folland}.  The minimax argument uses \cite{sion}, the Mellin discussion uses \cite{zagier}, and the orbit-averaging results use \cite{bailleul}.  The finite-field example uses Weil's curve point-count formula and Deligne's Frobenius eigenvalue bound \cite{weil,deligne}.

\begingroup
\small
\phantomsection
\addcontentsline{toc}{section}{References}
\bibliographystyle{unsrtnat}
\bibliography{GQRS_source_audited_final}
\endgroup

\end{document}